\documentclass[12pt,reqno]{amsart}
\usepackage{amsmath}
\usepackage{amsfonts}
\usepackage{amssymb}
\usepackage{amscd}
\usepackage{hyperref,color}

\renewcommand{\bot}{\perp}
\newtheorem{theorem}{Theorem}[section]
\newtheorem{corollary}[theorem]{Corollary}
\newtheorem{lemma}[theorem]{Lemma}
\newtheorem{proposition}[theorem]{Proposition}

\newtheorem{definition}[theorem]{Definition}
\newtheorem{remark}[theorem]{Remark}
\newtheorem{notation}[theorem]{Notation}

\newtheorem{example}[theorem]{Example}

\numberwithin{equation}{section}
\begin{document}
\title{Transverse geometric formality}
\author[G.~Habib]{Georges Habib}
\address{Lebanese University \\
Faculty of Sciences II \\
Department of Mathematics\\
P.O. Box 90656 Fanar-Matn \\
Lebanon\\ 
And Universit\'e de Lorraine\\
CNRS, IECL\\
54506 Nancy, France}
\email[G.~Habib]{ghabib@ul.edu.lb}
\author[K.~Richardson]{Ken Richardson}
\address{Department of Mathematics \\
Texas Christian University \\
Fort Worth, Texas 76129, USA}
\email[K.~Richardson]{k.richardson@tcu.edu}
\author[R.~Wolak]{Robert Wolak}
\address{Wydzial Matematyki i Informatyki UJ \\
Ul. prof. Stanislawa Lojasiewicza 6 \\
30-348 Krakow, Poland}
\email[R.~Wolak]{robert.wolak@uj.edu.pl}
\thanks{The authors acknowledge that the research cooperation was funded by the program Excellence Initiative – Research University at the Jagiellonian University in Krakow within the framework of the research group Reeb-Reinhart 2022. The second author acknowledges support of Alfried Krupp Wissenschaftskolleg.}
\date{May 17, 2024}

\begin{abstract}
A Riemannian metric on a closed manifold is said to be geometrically formal if the 
wedge product of any two harmonic forms is harmonic; equivalently, the interior product of any two harmonic forms is harmonic. Given a Riemannian foliation on a closed manifold, we say that a bundle-like metric is 
transversely geometrically formal if the interior product of any two basic harmonic forms 
is basic harmonic. In this paper, we examine the geometric and topological consequences of this condition.
\end{abstract}

\maketitle
\tableofcontents

\section{Introduction}

The notion of formality, as introduced in \cite{Sullivan1973DiffFormsTopMfldsBook} and explored further in
\cite{DeligneGriffithsMorganSullivan1975RealHomotKahlerMflds}, played a very important role in the study of the
topology of manifolds.   In particular, the authors demonstrated that compact K\"ahler manifolds are formal, thus the formality is an obstruction to the existence of  a K\"{a}hler structure. Moreover, there are examples of symplectic, non-K\"ahler solvmanifolds; cf. \cite%
{TralleOprea1997SymplMfldsNoKahler}. Formal solvmanifolds were investigated
in depth in \cite{Kasuya2013GeomFormSolvmflds}. 

In \cite{Kotschick2001OnProductsHarmFms} Kotschick proposed a finer notion
of formality for Riemannian manifolds: a compact Riemannian manifold $(M,g)$
is geometrically formal if the wedge product of any two harmonic forms is a
harmonic form. Therefore the space of harmonic forms is an algebra with the
wedge product, and according to the Hodge theorem it is a minimal model for
the cohomology of $M$. Thus any geometrically formal Riemannian manifold $%
(M,g)$ is formal. The investigation into the geometrical formality has been
continued by Kotschick and some other researchers resulting in numerous
interesting theorems (cf. \cite{BangertKatz2004_OptimalHarm1formsConstNorm},\cite{GrosjeanNagy2009OnTheCohomologyAlgGeomFormal},
 \cite{Kotschick2017GeometricFormNonnegScalarCurv}),
among them a fine classification of geometrically
formal Riemannian manifolds in low dimensions in \cite{BaerC2015GeometricFormal4MfldsNonnegSectCurv}. In \cite%
{NagyP2006OnTheLengthAndProduct}, the author gave a complete description of K%
\"{a}hler manifolds where all harmonic one-forms have constant length, among
other results.

Riemannian foliated manifolds (technically, Riemannian manifolds with
foliations endowed with bundle-like metrics) form a subclass of
Riemannian manifolds that appears naturally in many geometrical and applied
problems. The foliated aspect is best expressed by the study of the
transverse Riemannian geometry of the foliation. Moreover, one can
investigate the interplay between the transverse and global properties of
the foliated manifold. One of the tools is the basic cohomology  $H^*(M,\mathcal{F})$,   which in
the case of compact Riemannian foliated manifolds is finite dimensional, 
\cite{ElKacimiHector1986DecompositionHodge}, \cite%
{KamberTondeur1987deRhamHodgeTheory}, \cite%
{ParkRichardson1996_BasicLaplacian}. Additionally, the Hodge theory for the
basic cohomology was developed as well as the Poincar\'{e} duality
property, cf. \cite{KamberTondeur1984DualityThmsFoliations}.

Continuing this line of research we introduce the notion of transverse
geometric formality for a compact Riemannian foliated manifold $(M,g,%
\mathcal{F})$. We might like to say that $(M,g,\mathcal{F})$ is transversely
geometrically formal if any wedge product of two basic harmonic forms is
basic harmonic, but the presence of mean curvature causes issues with such a
definition; since the basic Laplacian does not commute with the basic Hodge
star-operator, under such a definition basic harmonic forms would not have
constant norm, which is a key property of formal metrics. Instead, the
appropriate definition of transverse geometric formality is that the
interior product of any two basic harmonic forms is also basic harmonic. If
the basic mean curvature is zero, then these two properties coincide, as
they do for formal metrics on Riemannian manifolds.

In Section \ref{PrelimSection} we recall, for the reader's convenience, some
definitions and properties of Riemannian foliated manifolds. The notion of
transverse geometric formality is introduced in Section \ref%
{TransvFormalitySection}, in which we provide the definition and investigate
the basic properties of general Riemannian foliated manifolds. 
We show an upper bound on the dimensions of basic cohomology satisfied by 
foliations admitting transversely formal metrics in
 Corollary \ref{basicCohomologyLiimitationCorollary}.
 In Theorem \ref{maxRankImpliesMinimalThm}, we consider a bundle-like metric
 on the foliation that satisfies transverse formality and has basic harmonic mean curvature.
In this case, if the dimension of a cohomology group is maximal, the foliation is necessarily minimal.
 
We next investigate the relation between geometric formality and transverse
geometric formality in Subsection \ref{FormalityTransverseFormalitySection}. 
For example, in Proposition \ref{formalImpliesTransverseFormalProp}, we show that for a Riemannian foliated
manifold with basic mean curvature and involutive normal bundle,
geometric formality implies transverse formality. 
Next, we study
transverse geometric formality of transversely Lie foliations, a class of
Riemannian foliations which, due to Molino's structure theorem, cf. \cite%
{Molino1988_RiemFolsBook}, are the basic building blocks of any compact
Riemannian foliated manifold.
We prove various results in this section, including Proposition \ref{LieFoliationProposition},
where we show that if a bundle-like metric is transversely formal and has basic harmonic
mean curvature with maximal $\dim H^1(M,\mathcal{F})$, then the foliation is a
$\mathbb{R}^q$-Lie foliation.
In Subsection \ref{TransvFormalityNotTransversePropSubsection}, we demonstrate that
geometric formality is not a transverse property in H\"{a}efliger's sense, and we show for example 
that taut foliations that are transversely formal must be minimal.

  Section \ref{FlowSection}   is dedicated to the detailed study of
one-dimensional Riemannian foliations, i.e. Riemannian flows, and their
subclasses, namely isometric flows, K-contact manifolds, and   Sasakian
manifolds.   Among other results, we prove in Theorem \ref{prop:minimalflowformal} 
that if a Riemannian flow is minimal and   transversely geometrically formal   and satisfies a Ricci curvature condition,   some
inequalities   are satisfied between the basic and ordinary cohomology.
In Theorem \ref{SasakiFormalImpliesTransFormalTheorem}, on closed   Sasakian manifold,  
geometric formality implies that the foliation defined by the Reeb vector field is 
transversely   geometrically formal.   We classify all one-dimensional transversely   geometrically formal   foliations on 
closed three-manifolds in Theorem \ref{classificationTheorem}.

In Section \ref{ExamplesSection}, we present some relevant examples. First,
we show that   $\mathbb{S}^{3}$   foliated by the Hopf fibration is transversely
geometrically formal. Then we investigate Carri\`{e}re's example of a
non-isometric flow on a 3-manifold and verify that it is transversely
geometrically formal. Finally, we construct some Riemannian foliations on a
solvmanifold and study their geometric formality. In the appendix we give an example   of a nontaut Riemannian foliation of codimension 4, with noninvolutive normal bundle carrying a transversely formal bundle-like metric.

\section{Preliminaries}\label{PrelimSection}

Let $\left( M,g\right) $ be a closed Riemannian manifold of dimension $m$,
endowed with a transversely oriented Riemannian foliation $\mathcal{F}$ of
codimension $q$ and dimension $p=m-q$. The foliation $\mathcal{F}$ is
given by an integrable subbundle $T\mathcal{F}$ of $TM$ of rank $p$.   Let $Q=TM\diagup T%
\mathcal{F}\cong \left( T\mathcal{F}\right) ^{\bot }$ be the normal bundle
of the foliation. In this paper we will in fact let $Q=\left( T\mathcal{F}%
\right) ^{\bot }$ to simplify exposition. We assume the metric $g$   to be   bundle-like, meaning that
the induced metric $g_{Q}$ of $g$ on $Q$ satisfies   the holonomy invariance property.   This
means 
\begin{equation*}
\mathcal{L}_{X}g_{Q}=0,
\end{equation*}%
for any $X\in \Gamma \left( T\mathcal{F}\right) $, where  $\mathcal{L}_X$
denotes the Lie derivative in the direction of $X$. Such metrics have been studied for decades; see 
\cite{Reinhart1959FolMfldsBundleLikeMetrics,Molino1988_RiemFolsBook,
Tondeur1997GeometryOfFoliations}.  The existence of bundle-like metrics assures the existence of a transverse Levi-Civita 
connection $\nabla $ on $\Gamma(Q)$,   i.e. a torsion-free connection compatible with the metric $g_{Q}$ extending the Bott partial connection; see \cite[Ch. 3]%
{Tondeur1997GeometryOfFoliations}.   This connection can be extended in the
usual way to act on  $\Gamma(\wedge^{\ast}Q^{\ast })$.  For simplicity we
will assume that our foliated manifolds are oriented and transversely
oriented.

\begin{notation}
We call the triple $\left( M,\mathcal{F},g\right) $ an oriented Riemannian
foliated manifold (ORFM) if $M$ is an oriented, closed Riemannian manifold
with metric $g$, and $\mathcal{F}$ is a transversely oriented Riemannian
foliation for which $g$ is bundle-like.
\end{notation}

Following \cite{Tondeur1997GeometryOfFoliations} we define the mean
curvature form $\kappa $   (or its dual vector $\kappa ^{\#}$)   as

\begin{equation*}
\kappa ^{\#}=\sum_{j=1}^{p}\pi \left( \nabla _{f_{j}}^{M}f_{j}\right) ,
\end{equation*}%
where  $\pi :TM\rightarrow Q$ is the projection, $\left\{ f_{j}\right\}_{j=1,\ldots,p} $
is a local orthonormal frame of $T\mathcal{F}$, and $\nabla ^{M}$ is the
Levi-Civita connection on $M$.

The characteristic form $\chi _{\mathcal{F}}$ of the foliation $\mathcal{F}$
is the leafwise volume form and is defined as follows:

\begin{equation*}
\chi _{\mathcal{F}}(Y_{1},...,Y_{p})=\det (g(Y_{i},f_{j})_{i,j=1,...,p}),
\end{equation*}
where $Y_{1},...,Y_{p}$ are any vector fields on $M$.

The characteristic form $\chi_{\mathcal{F}}$ and the mean curvature form $%
\kappa$ are related by the Rummler's formula, \cite{RummlerH1979Qulelques}:

\begin{equation*}
d\chi _{\mathcal{F}}=-\kappa \wedge \chi _{\mathcal{F}}+\varphi _{0},
\end{equation*}%
where $\varphi _{0}$ satisfies the property $\chi _{\mathcal{F}}\lrcorner
\varphi _{0}=0$; here, $\lrcorner $ denotes interior product of differential
forms. The mean curvature one-form $\kappa $ depends on the choice of the
bundle-like metric. The normal bundle $Q$ is involutive if and only if $%
\varphi _{0}=0$.

On any Riemannian foliated manifold $(M,g,\mathcal{F})$, the set of all
basic $r$-forms ($0\leq r\leq q$) is
\begin{equation*}
\Omega^{r}(M,\mathcal{F})=\{\alpha \in   \Omega^{r}(M):  X\lrcorner \alpha =X\lrcorner d\alpha =0 ~%
\text{for all vectors }\,X\in \Gamma (T\mathcal{F})\},
\end{equation*}%
which is a subcomplex of the differential forms  $\Omega^{r}(M)$  on $M$.   The basic cohomology is then defined as being the cohomology associated to the exterior differential restricted to basic forms, that is 
$$H^{r}(M,\mathcal{F})=\frac{{\rm ker}(d:\Omega^r(M,\mathcal{F})\to \Omega^{r+1}(M,\mathcal{F}))}{{\rm im}(d:\Omega^{r-1}(M,\mathcal{F})\to \Omega^r(M,\mathcal{F}))},$$ 
for $0\leq r\leq q$. It is shown in \cite{ElKacimiHector1986DecompositionHodge}, \cite%
{KamberTondeur1987deRhamHodgeTheory}, \cite%
{ParkRichardson1996_BasicLaplacian} that the basic cohomology is always finite dimensional but does not necessarily satisfy Poincar\'e duality (see \cite{Carriere1984FlotsRiem} for an example).  

The restriction of the
bundle-like metric to the normal bundle of the foliation of the   oriented   Riemannian
manifold $(M,g,\mathcal{F})$ defines the basic Hodge star operator  $\Bar{\ast}$, cf. 
\cite{Tondeur1997GeometryOfFoliations},

\begin{equation*}
\Bar{\ast}\colon \Omega^{r}(M,\mathcal{F})\rightarrow \Omega^{q-r}(M,\mathcal{F}).
\end{equation*}
  This operator is related to the usual Hodge $\ast $-operator on $(M,g)$ by the formula
\begin{equation*}
\Bar{\ast}\alpha =(-1)^{p(q-r)}\ast (\alpha \wedge \chi _{\mathcal{F}})
\end{equation*}%
for any $\alpha \in \Omega^{r}(M,\mathcal{F}) $.  The pointwise inner product   between basic forms   is defined in the usual way as;
if $\alpha ,\beta \in \Omega^{r}(M,\mathcal{F}) $ we have%
\begin{equation*}
\left( \alpha ,\beta \right) =\overline{\ast }\left( \alpha \wedge \overline{%
\ast }\beta \right) .
\end{equation*}%
%
In this case, the standard scalar product satisfies 
\begin{equation*}
\langle \alpha ,\beta \rangle =\int_{M}\alpha \wedge \Bar{\ast}\beta \wedge
\chi _{\mathcal{F}}.
\end{equation*}%

A Riemannian foliation on a compact manifold is said to be taut if there
exists a Riemannian metric which makes all its leaves minimal submanifolds.
Tautness is characterized by the nonvanishing of the top dimensional basic
cohomology, i.e., $H^{q}(M,\mathcal{F})\neq 0$ (see \cite{MasaX1992DualityMinRiemFols}).

Let $L^2(\Omega(M))$ and  $L^2(\Omega(M,\mathcal{F}))$ denote the   closures of $\Omega(M)$ and $\Omega(M,\mathcal{F})$  with respect to the $L^2$ inner product on forms $\langle
.\,,.\,\rangle $,   and let $P_b:L^2(\Omega(M))\to L^2(\Omega(M,\mathcal{F}))$ denote the orthogonal projection (see \cite{ParkRichardson1996_BasicLaplacian}). We denote $\kappa _{b}=P_b\kappa$. By \cite{AlvarezLopez1992BasicComponentMeanCurvature}, $%
\kappa _{b}$ is always closed, and it defines a cohomology class $\left[
\kappa _{b}\right] \in H^{1}\left( M,\mathcal{F}\right) $ that is
independent of the choice of   a bundle-like metric.   By \cite%
{Dominguez1998FinitenessTensenessThmsRiemFols}, there exists a bundle-like
metric $g^{\prime }$ whose mean curvature is basic ($\kappa =\kappa _{b}$), such that $g$ and $g^{\prime}$ define the same metric on   $Q$.   Another property that characterises tautness of a Riemannian foliation is that $[\kappa _{b}]=0$ (see \cite{MasaX1992DualityMinRiemFols}).

The formal adjoint $\delta _{b}$ of $d$ in the complex  $\Omega^{*}(M,\mathcal{F})$ 
with respect to the scalar product $\langle .\,,.\,\rangle $ is the operator 
\begin{equation*}
\delta _{b}=(-1)^{q(r+1)+1}\Bar{\ast}%
(d-\kappa _{b}\wedge )\Bar{\ast}:  \Omega^{r}(M,\mathcal{F}) \rightarrow  \Omega^{r-1}(M,\mathcal{F}).
\end{equation*}%
The basic Laplacian is the operator on basic forms defined by 
\begin{equation*}
\Delta _{b}=\delta _{b}d+d\delta _{b}.
\end{equation*}

A basic form $\alpha $ is called \textbf{basic harmonic} if and only if $%
\Delta _{b}\alpha =0$; equivalently, $d\alpha=0$ and $\delta_b\alpha=0$.

The transverse volume form $\nu $ is a basic $q$-form; it satisfies  $\delta
_{b}\nu =\kappa _{b}\lrcorner \nu=\Bar{\ast}\kappa _{b}$.   In particular, the transverse volume form $%
\nu $ is basic harmonic if and only if the basic component of mean curvature is zero. This is certainly the case if the foliation is minimal, but general bundle-like metrics on a taut Riemannian foliation do not necessarily satisfy this condition.

The basic Hodge theorem was proved in special cases by \cite%
{ElKacimiHector1986DecompositionHodge}, \cite%
{KamberTondeur1987deRhamHodgeTheory} and in general by \cite%
{ParkRichardson1996_BasicLaplacian}. In particular, each basic cohomology
class contains a unique basic harmonic form.

Let us introduce the twisted differential
\begin{equation*}
d_{\kappa _{b}}:=d-\kappa _{b}\wedge.
\end{equation*}
Since $\kappa_b$ is closed, the twisted differential satisfies $d^2_{\kappa _{b}}=0$ and, therefore, defines a chain complex.  We denote by $H_{\kappa_b}^*(M,\mathcal{F})$ the corresponding cohomology.  The formal adjoint of $d_{\kappa _{b}}$ is $\delta_{\kappa _{b}}:=\delta_b-\kappa _{b}\lrcorner$.   The two adjoint operators   $\delta_b$ and $\delta_{\kappa_b}$  are related by
\begin{equation*}
\Bar{*}\delta_b=\delta_{\kappa_b}\Bar{*}.
\end{equation*} Hence, we introduce  the twisted Laplacian operator    as   \begin{equation*}
\Delta_{\kappa_b} := d_{\kappa_b}\delta_{\kappa_b} + \delta_{\kappa_b} d_{\kappa_b}.
\end{equation*}  

 As for the basic cohomology, we can show that there is a basic Hodge theorem for the twisted Laplacian and that each cohomology class in $H_{\kappa_b}^*(M,\mathcal{F})$ contains a unique basic  \textbf{$\kappa_b$-harmonic} form. Here, a basic form $\alpha$ is called  $\kappa_b$-harmonic if $\Delta_{\kappa_b}\alpha=0$; equivalently, $%
d\alpha =\kappa _{b}\wedge \alpha $ and $\delta _{b}\alpha =\kappa
_{b}\lrcorner \alpha $.  

For the basic cohomology of compact Riemannian foliated manifolds we have
the following twisted duality theorem (cf. \cite[Theorem 7.54]%
{Tondeur1997GeometryOfFoliations}, with modifications from \cite{ParkRichardson1996_BasicLaplacian} and \cite{AlvarezLopez1992BasicComponentMeanCurvature}):

\begin{theorem}
Let $(M,{\mathcal{F}},g)$ be an ORFM. The pairing $\alpha \otimes \beta \mapsto \int_{M}\alpha \wedge
\beta \wedge \chi _{\mathcal{F}}$ induces a nondegenerate pairing
\begin{equation*}
H^r(M,{\mathcal{F}}) \otimes H_{\kappa_b}^{q-r}(M,{\mathcal{F}})
\rightarrow \mathbb{R}
\end{equation*}
of finite dimensional spaces.
\end{theorem}

  It is not difficult to check that the transversal Hodge star $\overline{\ast }$ operator maps basic harmonic $r$-forms
to basic $\kappa _{b}$-harmonic $\left( q-r\right) $-forms, and vice-versa. This is a consequence of the fact that $\Delta _{b}\overline{\ast }=%
\overline{\ast }\Delta _{\kappa _{b}}$ and $\Delta _{\kappa _{b}}\overline{%
\ast }=\overline{\ast }\Delta _{b}$; see \cite[proof of Theorem 7.54]%
{Tondeur1997GeometryOfFoliations}, with modifications from \cite[Corollary
3.3 and proof]{AlvarezLopez1992BasicComponentMeanCurvature} or \cite[%
Proposition 2.2]{ParkRichardson1996_BasicLaplacian}. Therefore, by the basic Hodge theorem for $\Delta_b$ and $\Delta_{\kappa_b}$, we have that 
$$H^*(M,{\mathcal{F}})\simeq  H_{\kappa_b}^{q-*}(M,{\mathcal{F}}).$$
This also explains the reason why Poincar\'e duality does not hold in general for the basic cohomology.

  Finally, we end with the following result that we will use in the sequel. 
\begin{lemma}
\label{basicHarmonicOneFormsAreHarmonicLemma}If $\left( M,\mathcal{F,}%
g\right) $ has basic mean curvature, then every basic harmonic one-form is
harmonic.
\end{lemma}

\begin{proof}
Let $p=\dim \mathcal{F}$. If $\alpha $ is a basic harmonic one-form, by \cite%
[Proposition 2.4]{ParkRichardson1996_BasicLaplacian} we have that $d\alpha
=0 $ and 
\begin{eqnarray*}
0 &=&\delta _{b}\alpha =\left( \delta +\left( \kappa _{b}-\kappa \right)
\lrcorner +\left( -1\right) ^{p}\varphi _{0}\lrcorner \circ \chi_\mathcal{F} \wedge
\right) \alpha \\
&=&\left( \delta +0+\left( -1\right) ^{p}\varphi _{0}\lrcorner \circ \chi_\mathcal{F}
\wedge \right) \alpha =\delta \alpha ,
\end{eqnarray*}%
since $\varphi _{0}\lrcorner \left( \chi_\mathcal{F} \wedge \cdot \right) $ is always
zero on  basic one-forms.
\end{proof}

\section{Transverse formality}\label{TransvFormalitySection}

  In this section, we introduce the notion of transverse geometric formality of an ORFM both for taut and nontaut Riemannian foliations. The main problem in the nontaut case is the fact that the basic Hodge star operator does not commute with the basic Laplacian. An equivalent definition using the interior product of forms avoids this awkwardness. We investigate the basic cohomology of a transversely formal ORFM and their influence on the topology of the foliated manifold. Particular attention is paid to transversely Lie foliations. Then the relation between geometric formality and transverse formality is studied. Finally, it is shown that transverse formality is not a transverse property.  

\begin{definition}
Let $\left( M,\mathcal{F},g\right) $ be an ORFM. We say that $\left( M,%
\mathcal{F},g\right) $ is \textbf{transversely formal (or transversely
geometrically formal)} if the wedge product of any  basic harmonic form  and
basic $\kappa _{b}$-harmonic form is basic $\kappa _{b}$-harmonic.
\end{definition}

\begin{definition}
Let $\left( M,\mathcal{F},g\right) $ be  an  ORFM. We say that $\left( M,%
\mathcal{F},g\right) $ is \textbf{transversely }$r$-\textbf{formal} \textbf{%
(or transversely geometrically }$r$-\textbf{formal)} for $%
0\leq r\leq q$ if the wedge product of
any  basic harmonic  $r$-form and basic $\kappa _{b}$-harmonic $\left(
q-r\right) $-form is basic $\kappa _{b}$-harmonic.
\end{definition}

\begin{remark}
Note that in \cite{OrneaPilca2011RemProdHarmForms}, $r$-formality is the
condition that the wedge product of any two harmonic $r$-forms is harmonic.
In our foliation case, it turns out to be more convenient to use the
definition above, in light of Lemma \ref{constantInnerProdLemma}.
\end{remark}

We note that transverse formality and formality are distinct;   an ORFM   
may satisfy one property but not the other. See Section \ref%
{ExamplesSection}, which contains several interesting examples of   Riemannian  
foliations that satisfy transverse formality.

\subsection{General results concerning transverse formality}

In this   subsection,   we prove the fundamental properties of transversely formal
metrics.

\begin{lemma}
\label{constantInnerProdLemma}Suppose that $\left( M,\mathcal{F},g\right) $
is transversely $r$-formal. Then the pointwise inner product of any two
basic harmonic $r$-forms is constant. In particular, basic harmonic $r$%
-forms have constant length. The same is true for basic $\kappa _{b}$%
-harmonic $\left( q-r\right) $-forms.
\end{lemma}

\begin{proof}
If the $r$-forms $\alpha ,\beta $ are basic harmonic and $\nu $ is the
transverse volume $q$-form, then the pointwise inner product satisfies 
\begin{equation*}
\left( \alpha ,\beta \right) \nu =\alpha \wedge \overline{\ast }\beta ,
\end{equation*}%
and $\overline{\ast }\beta $ is basic $\kappa _{b}$-harmonic. Since   $%
H_{\kappa _{b}}^{q}\left( M,\mathcal{F}\right)  \cong H^{0}\left( M,\mathcal{F}\right)\cong\mathbb{R}$  is generated by the
basic $\kappa _{b}$-harmonic form $\nu $. Thus, $\left( \alpha ,\beta
\right) $ is constant (and also $\left( \overline{\ast }\alpha ,\overline{%
\ast }\beta \right) $ is constant).
\end{proof}

\begin{corollary}
\label{basicCohomologyLiimitationCorollary}Let  $\mathcal{F}$ 
be a codimension $q$ foliation of a compact manifold $M$ that admits a
transversely $r$-formal bundle-like metric $g$. Then the basic cohomology satisfies \newline$%
\dim H^{r}\left( M,\mathcal{F}\right) \leq \binom{q}{r}$ for $%
0\leq r\leq q$.
\end{corollary}
\begin{proof}
From the lemma above, for the transversely formal metric, the number of
linearly independent harmonic forms of degree  $r$ is at most the rank of $%
\wedge ^{r}Q^{\ast }=\binom{q}{r}$. 
\end{proof}

\begin{remark}
For ordinary $n$-dimensional closed manifolds, the existence of a
geometrically formal metric implies that $\dim H^{1}\left( M\right)
 \neq n-1$ (see \cite[Theorem 6]{Kotschick2001OnProductsHarmFms}).
However, for particular transversely formal foliations of codimension $q$, the   dimension   of
$H^{1}\left( M,\mathcal{F}\right)$ might be equal to $q-1$ (see Corollary \ref{basicfirstCohomologynontautCorollary}); the
proof does not generalize because of the interaction with the mean
curvature. See for example Section \ref{Carriere} for an example of a
codimension $2$ foliation that has a transversely formal metric where $\dim H^{1}\left( M,\mathcal{F}\right)=1$. If, however, the foliation is taut,
we do have $\dim H^{1}\left( M,\mathcal{F}\right)  \neq q-1$
(see Corollary \ref{q-1Corollary}).
\end{remark}

 As equivalent definition of transverse formality is given in the next proposition:
\begin{proposition}
The Riemannian foliated manifold $\left( M,\mathcal{F},g\right) $ is
transversely formal if and only if the interior product of any two basic
harmonic forms is basic harmonic. The same is true for two basic $\kappa
_{b} $-harmonic forms.
\end{proposition}

\begin{proof}
Suppose that $\left( M,\mathcal{F},g\right) $ is transversely formal. If $%
\alpha ,\beta $ are two basic harmonic forms, then $\overline{\ast }\beta $
is basic $\kappa _{b}$-harmonic, so by transverse formality $\alpha \wedge 
\overline{\ast }\beta $ is basic $\kappa _{b}$-harmonic. Thus, $\alpha
\lrcorner \beta =\pm \overline{\ast }\left( \alpha \wedge \overline{\ast }%
\beta \right) $ is basic harmonic. A similar proof works for two basic $%
\kappa _{b}$-harmonic forms.\newline
Conversely, suppose that $\left( M,\mathcal{F},g\right) $ has the required
property, and suppose that $\alpha $ is basic harmonic and $\beta $ is basic 
$\kappa _{b}$-harmonic. Then $\overline{\ast }\beta $ is basic harmonic, and
thus $\alpha \lrcorner \overline{\ast }\beta =\pm \overline{\ast }\left(
\alpha \wedge \beta \right) $ is basic harmonic, so that $\alpha \wedge
\beta $ is basic $\kappa _{b}$-harmonic. A similar proof works for the $%
\kappa _{b}$-harmonic case.
\end{proof}

In the following, we will study the case where the foliation is transversaly formal and that   $\dim  H^r(M,\mathcal{F})=\binom{q}{r}$   for some $r$. We will see that this is a restriction on the geometry of the foliation. For this, we need several lemmas:



\begin{lemma}
\label{nablaFactsProp}Suppose that $\left( M,\mathcal{F},g\right) $ is
transversely formal, and $\kappa $ is basic harmonic. If $\alpha $ is basic harmonic $r$-form, then 
$$\nabla _{\kappa ^{\#}}\alpha+\sum_{i=1}^q\nabla _{e_{i}}\kappa \wedge \left( e_{i}\lrcorner \alpha \right) =0,$$
where $\left\{e_{i}\right\} _{i=1,\ldots,q}$ be a local orthonormal frame of $\Gamma(Q)$.



\end{lemma}

\begin{proof}
With the assumptions, we have that $d\alpha =0$ and $\kappa^{\#} \lrcorner \alpha $
is basic harmonic by Proposition \ref{nablaFactsProp}, thus it is also closed. Then 
\begin{eqnarray*}
\mathcal{L}_{\kappa ^{\#}}\alpha =d\left( \kappa^{\#} \lrcorner \alpha \right)
+\kappa^{\#} \lrcorner \left( d\alpha \right)=0.
\end{eqnarray*}
Recall the Lie derivative along any vector field $X$ on differential $r$-forms can be related to the Levi-Civita connection on $M$ 
by the formula $\mathcal{L}_{X}=\nabla^M_X+S_X^{[r]}$ where $S_X^{[r]}$ is the canonical extension of the endomorphism $S=\nabla^M X$ to differential forms defined by 
$$(S_X^{[r]}\omega)(X_1,\dots,X_r)=\sum_{j=1}^r\omega(X_1,\dots,\nabla^M_{X_j}X,\dots, X_r)$$
for any differential $r$-form $\omega$ and $X_1,\ldots, X_r\in TM$. Equivalently, this extension is equal to $$S_X^{[r]}=\sum_{l=1}^pf_l\wedge S(f_l)\lrcorner+\sum_{i=1}^q e_i\wedge S(e_i)\lrcorner,$$ 
where $\{f_l\}_{l=1,\ldots,p}$ is a local orthonormal frame of $T\mathcal{F}$ and $\{e_i\}_{i=1,\ldots, q}$ is a local orthornormal frame of $Q$.
Hence for $X=\kappa ^{\#}$, we get that $\nabla^M_{\kappa ^{\#}}\alpha+S_{\kappa ^{\#}}^{[r]}\alpha=0$. Applying this last identity to sections in $Q$, we deduce the required identity, since the form $\alpha$ is basic and $S=\nabla^M\kappa$ is a symmetric endomorphism as $\kappa$ is closed. 
\end{proof}

\begin{lemma}
\label{nablaKappaThm}Suppose that the ORFM $\left( M,\mathcal{F}%
,g\right) $ is transversely formal, and suppose that the mean curvature form 
$\kappa $ is basic harmonic. Then for any two basic harmonic forms $\alpha $%
, $\beta $, the pointwise inner product $\left( \nabla _{\kappa ^{\#}}\alpha
,\beta \right) $ is identically zero. The same fact is true for two basic $%
\kappa $-harmonic forms.
\end{lemma}


\begin{proof} Let $\alpha$ and $\beta$ be two basic harmonic forms. From Lemma \ref{nablaFactsProp}, we have that
\begin{equation*}
\nabla _{\kappa ^{\#}}\alpha +\sum_{i=1}^q\nabla _{e_{i}}\kappa \wedge e_{i}\lrcorner
\alpha =0.
\end{equation*}
We take the pointwise inner product with $\beta $ to get%
\begin{equation}
\left( \nabla _{\kappa ^{\#}}\alpha ,\beta \right) +\sum_{i=1}^q\left( e_{i}\lrcorner
\alpha ,\nabla _{e_{i}}\kappa \lrcorner \beta \right) =0.  \label{eq1}
\end{equation}%
Switching the roles of $\alpha $ and $\beta $, we get%
\begin{equation}
\left( \nabla _{\kappa ^{\#}}\beta ,\alpha \right) +\sum_{i=1}^q\left( e_{i}\lrcorner
\beta ,\nabla _{e_{i}}\kappa \lrcorner \alpha \right) =0.  \label{eq2}
\end{equation}%
Note that $\sum_{i=1}^q\left( e_{i}\lrcorner \alpha ,\nabla _{e_{i}}\kappa \lrcorner
\beta \right) =\sum_{i=1}^q\left( e_{i}\lrcorner \beta ,\nabla _{e_{i}}\kappa \lrcorner
\alpha \right) $, because $\nabla \kappa $ is a symmetric two-tensor as $%
\kappa $ is closed. 
Thus, substracting (\ref{eq1}) and (\ref{eq2}), we get 
\begin{equation*}
\left( \nabla _{\kappa ^{\#}}\alpha ,\beta \right) =\left( \nabla _{\kappa
^{\#}}\beta ,\alpha \right) .
\end{equation*}%
But, since $\left( \alpha ,\beta \right) $ is constant as the metric is transversely formal, we get that 
$$0=\kappa
^{\#}\left( \alpha ,\beta \right) =\left( \nabla _{\kappa ^{\#}}\alpha
,\beta \right) +\left( \alpha ,\nabla _{\kappa ^{\#}}\beta \right),$$ 
so that $\left( \nabla _{\kappa ^{\#}}\alpha ,\beta \right) =0$ which is the statement of the lemma. The last part follows from the fact that $\overline{\ast }$ maps the space of basic $\kappa $-harmonic forms
isometrically to the space of basic harmonic forms  and that $%
\overline{\ast }$ is a transversal isometry. 
\end{proof}

\begin{theorem}
\label{maxRankImpliesMinimalThm}Suppose that the ORFM $\left( M,\mathcal{F}%
,g\right) $ is transversely formal and has basic harmonic mean curvature. If 
$\dim  H^{r}\left( M,\mathcal{F}\right)  =\binom{q}{r}$ for
some $r$ with $0<r\leq q$, then the foliation is minimal.

\begin{proof}
If $r=q$, $\dim  H^{r}\left( M,\mathcal{F}\right)  =1$, so the
foliation is taut. Since the mean curvature is basic harmonic, it must be
zero. Now suppose $\dim  H^{r}\left( M,\mathcal{F}\right) =%
\binom{q}{r}$ for some $r$ with $0<r<q$. By Lemma \ref{nablaKappaThm}, $%
\left( \nabla _{\kappa ^{\#}}\alpha ,\beta \right) =0$ for all basic
harmonic $r$-forms $\alpha $, $\beta $. Since $\dim  H^{r}\left( M,%
\mathcal{F}\right)  =\binom{q}{r}=\mathrm{rank}\wedge ^{r}Q^{\ast }$,
then $\nabla _{\kappa ^{\#}}\alpha =0$. Here we are using the fact that the
pointwise inner product of basic harmonic $k$-forms is constant, so that if
a set of forms is linearly dependent at one point, then it is linearly
dependent globally. Now Lemma \ref{nablaFactsProp} gives that $\sum_{i=1}^q\nabla _{e_{i}}\kappa \wedge
e_{i}\lrcorner \alpha =0$ for all basic harmonic $r$-forms $\alpha $, which
means in fact that $\sum_{i=1}^q\nabla _{e_{i}}\kappa \wedge e_{i}\lrcorner \alpha =0$
for all $\alpha \in \wedge ^{r}Q^{\ast }$. Since the extension  $S_\kappa^{[r]}$  of the symmetric endomorphism $S=\nabla\kappa $ is identically zero, the endomorphism $S$ must itself be zero which implies $\nabla \kappa =0$. But then 
\begin{eqnarray*}
0 &=&\delta _{b}\kappa =-e_{i}\lrcorner \nabla _{e_{i}}\kappa +\left\vert \kappa \right\vert
^{2}=\left\vert \kappa \right\vert ^{2}
\end{eqnarray*}%
Then $\kappa $ is identically zero.
\end{proof}
\end{theorem}

As a direct consequence, we get the following result  
\begin{corollary} \label{basicfirstCohomologynontautCorollary} Suppose that the ORFM $\left( M,\mathcal{F}%
,g\right) $ is transversely formal and has basic harmonic mean curvature. If the foliation is nontaut, then $1\leq {\rm dim}H^1(M,\mathcal{F})\leq q-1$.
\end{corollary}
\begin{proof} 
Since the foliation is nontaut and the mean curvature is basic harmonic, we get that ${\rm dim}H^1(M,\mathcal{F})\geq 1$. Now, Corollary \ref{basicCohomologyLiimitationCorollary} gives that ${\rm dim}H^1(M,\mathcal{F})\leq q$. Since again the mean curvature cannot be zero, Theorem \ref{maxRankImpliesMinimalThm} implies that ${\rm dim}H^1(M,\mathcal{F})<q$. Hence, we deduce the statement.
\end{proof}

\begin{proposition}
Let $\left( M,\mathcal{F},g\right) $ be an ORFM with basic mean curvature
such that there exists a basic harmonic one-form $\alpha $ that has unit
length. Then there exists a codimension one minimal Riemannian foliation $%
\mathcal{F}_{\alpha }$ on $M$ such that $\mathcal{F}$ saturates the leaves
of $\mathcal{F}_{\alpha }$ and the universal cover $\left( \widetilde{M},%
\widetilde{\mathcal{F}_{\alpha }},\widetilde{g}\right) $ is a product
codimension one Riemannian foliation.
\end{proposition}

\begin{proof}
We are given that $\alpha $ is a basic harmonic one-form of constant length $%
1$, so that it is in fact also harmonic (Lemma \ref%
{basicHarmonicOneFormsAreHarmonicLemma}). Then $\ast \alpha $ is also a
harmonic form of constant length, and we have $d \alpha=0$
and $d\left( \ast \alpha \right) =0$. Then $\alpha $ is the characteristic
form of the foliation given by the flow of the vector field $\alpha ^{\#}$,
and $\ast \alpha $ is the characteristic form of the normal bundle. By
Rummler's formula, the $d\alpha =0$ implies the flow is minimal\ (geodesic)
and the normal bundle is involutive and defines a Riemannian foliation $%
\mathcal{F}_{\alpha }$ of codimension one. By $d\left( \ast \alpha \right)
=0 $, we see that the foliation $\mathcal{F}_{\alpha }$ is minimal. The form 
$\alpha $ generates the basic cohomology $H^{\ast }\left( M,\mathcal{F}%
_{\alpha }\right) $, and this cohomology injects into the cohomology of $M$.
By \cite[Corollary 3.3]{BlumenthalHebda1983_deRham} and \cite[Section 5]%
{BlumenthalHebda1984_Ehresmann}, the metric and foliations lift to the
universal cover $\left( \widetilde{M},\widetilde{\mathcal{F}},\widetilde{%
\mathcal{F}}_{\alpha },\widetilde{g}\right) $ so that the lifted $\widetilde{%
\mathcal{F}}_{\alpha }$ is a product foliation, and $T\widetilde{\mathcal{F}}%
\subseteq T\widetilde{\mathcal{F}}_{\alpha }$.
\end{proof}

\begin{remark}
In the case of transverse formality, the proposition above produces $k$
orthogonal codimension one minimal Riemannian foliations $\mathcal{F}%
_{\alpha _{j}}$ for an orthonormal set of basic harmonic one-forms $\left\{
\alpha _{j}\right\} $, where $k=\dim H^{1}\left( M,\mathcal{F}\right) $, such that the original foliation $\mathcal{F}$ is the contained
the intersection of all the $\mathcal{F}_{\alpha _{j}}$.
\end{remark}

\begin{proposition}
Let $\left( M,\mathcal{F},g\right) $ be a nontaut ORFM, such that the metric 
$g$ is transversely formal and $\kappa $ is basic harmonic (thus harmonic).
Then $H=\kappa ^{\#}$ generates a geodesic flow such that $\nabla _{\bullet
}H$ is symmetric and the normal bundle to $H\ $is integrable and is the
tangent bundle to a minimal Riemannian foliation that contains $\mathcal{F}$
as a subfoliation.
\end{proposition}

\begin{proof}
Since $\left\vert \kappa \right\vert ^{2}$ is constant, $H=\kappa ^{\#}$ has
constant norm and generates a geodesic flow. Since $d\kappa =0=d(\ast \kappa)$%
, by Rummler's formula normal bundle to $H$ is integrable and is the tangent
bundle to a minimal Riemannian foliation.
\end{proof}

\subsection{Formality versus transverse formality}\label{FormalityTransverseFormalitySection}
The main results show that geometric formality of the Riemannian manifold coupled with the assumption that the foliation has  basic mean curvature in some special cases implies transverse formality. 

\begin{proposition} \label{formalImpliesTransverseFormalProp}
Suppose $\left( M,\mathcal{F},g\right) $ is an ORFM with basic mean
curvature and integrable normal bundle. Assume that $\left( M,g\right) $ is
geometrically formal. Then $\left( M,\mathcal{F},g\right) $ is transversely
formal. If moreover the foliation is nontaut, the pointwise inner product of a basic 
$\kappa $-harmonic form and a basic harmonic form is always zero.
\end{proposition}
\begin{proof}
 Since the mean curvature is basic (i.e. $\kappa =\kappa _{b}$) and the
normal bundle is integrable, every basic harmonic form is also harmonic,
because the basic codifferential is a restriction of the ordinary
differential $\delta $ (see \cite[Proposition 2.4]%
{ParkRichardson1996_BasicLaplacian}). Also conversely, any harmonic form which is basic is a basic harmonic form. Let $\alpha $ be a basic $\kappa $%
-harmonic $r$-form  and let $\beta $ be a basic harmonic $s$-form. Since $\left( M,g\right) $ is 
geometrically formal, $\beta \lrcorner \left( \overline{\ast }\alpha \right) 
$ is harmonic. As $\beta $ is basic, we have $\beta \lrcorner \left( \overline{%
\ast }\alpha \right) =\pm \overline{\ast }%
\left( \beta \wedge \alpha \right) $ and, thus, $\overline{\ast }\left( \alpha
\wedge \beta \right) $ is harmonic. But this last form is also basic, hence it becomes basic harmonic, so that $\alpha
\wedge \beta $ is a basic $\kappa $-harmonic. Thus, $\left( M,\mathcal{F},g\right) $
is transversely formal. In addition, if $r=s$, then $\overline{\ast }\alpha
\wedge \beta =\left( \alpha ,\beta \right) \nu $ is harmonic and thus has
constant length. Then $\left( \alpha ,\beta \right) $ is a constant. If the
foliation is nontaut, then $\nu $ is not harmonic, since $\delta \nu=\delta_b\nu=\kappa
\lrcorner \nu $, so we have $\left( \alpha ,\beta \right) =0$.
\end{proof}

\begin{remark}
Note that it is possible that a   bundle-like metric   on a Riemannian foliation is both
transversely formal and geometrically formal, but where the normal bundle
is not integrable. See, for example, the Hopf fibration in Section \ref%
{HopfFibrationExample}.
\end{remark}

\begin{lemma}
\label{1-FormalityTheorem}If the ORFM $\left( M,\mathcal{F},g\right) $ is  geometrically formal  and has basic mean curvature, then the wedge product of any two
basic harmonic $1$-forms is basic harmonic.
\end{lemma}

\begin{proof}
If $\alpha $ and $\beta $ are basic harmonic $1$-forms, then they are also
harmonic, because $\delta _{b}$ is a restriction of $\delta $ on $1$-forms
(see \cite[Proposition 2.4]{ParkRichardson1996_BasicLaplacian}). Then $%
\alpha \wedge \beta $ is harmonic, since $M$ is formal. On the other hand,
from \cite[Proposition 2.4]{ParkRichardson1996_BasicLaplacian},   we have  
\begin{equation*}
\delta _{b}\left( \alpha \wedge \beta \right) =\delta \left( \alpha \wedge
\beta \right) +\left( -1\right) ^{\dim \mathcal{F}}\varphi _{0}\lrcorner
\left( \chi _{F}\wedge \alpha \wedge \beta \right).
\end{equation*}%
Since $\delta \left( \alpha \wedge \beta \right) =0$ and $\varphi
_{0}\lrcorner \left( \chi _{F}\wedge \alpha \wedge \beta \right) $ is
orthogonal to basic forms and is thus $0$, we deduce $\delta _{b}\left( \alpha
\wedge \beta \right) =0$. Therefore $\alpha \wedge \beta $ is basic harmonic.
\end{proof}

\begin{theorem}
\label{codim2One-formalProp}Let $\left( M,\mathcal{F},g\right) $ be an ORFM
such that $\left( M,g\right) $ is  geometrically formal  and $\mathcal{F}$ has codimension
two with basic mean curvature. Then $\left( M,\mathcal{F},g\right) $ is
transversely formal. If moreover the foliation is nontaut, then ${\rm dim}H^1(M,\mathcal{F})=1$. 
\end{theorem}

\begin{proof}
Suppose $\alpha $ is a basic harmonic $1$-form, and $\beta $ is a basic $%
\kappa $-harmonic $1$-form. Then $\overline{\ast }\beta $ is a basic
harmonic $1$-form. Thus, $\alpha $ and $\overline{\ast }\beta $ are harmonic
forms. Since the metric is formal, we have $(\alpha ,\ast \beta )=\left(
\alpha ,\overline{\ast }\beta \right) $ is constant. But $\alpha \wedge
\beta =\pm \left( \alpha ,\overline{\ast }\beta \right) \nu $, so since $%
\left( \alpha ,\overline{\ast }\beta \right) $ is constant and $\nu $ is $%
\kappa $-harmonic, $\alpha \wedge \beta $ is $\kappa $-harmonic. Hence $M$
is transversely formal.  By Corollary \ref{basicCohomologyLiimitationCorollary} we have ${\rm dim} H^1(M,\mathcal{F})\leq 2$. Notice here $H^1(M,\mathcal{F})$ cannot be zero since the foliation is nontaut. We shall now prove that ${\rm dim}H^1(M,\mathcal{F})=2$ cannot occur. Assume it were the case and let us denote by $\{\alpha_1,\alpha_2\}$ a basis of $H^1(M,\mathcal{F})$; then by Lemma \ref{1-FormalityTheorem} one would get that $\alpha_1\wedge\alpha_2$ is basic harmonic, hence belonging to $H^2(M,\mathcal{F})=0$, since the foliation is nontaut. Thus, $\alpha_1= c\alpha_2$ for some $c\in\mathbb{R}$, a contradiction. Therefore ${\rm dim}H^1(M,\mathcal{F})=1$.
\end{proof}

\begin{remark} From Corollary \ref{basicfirstCohomologynontautCorollary}, one can see that for codimension $2$ nontaut transversely formal foliations, we have that ${\rm dim}H^1(M,\mathcal{F})=1$. But this is true under the extra assumption that $\kappa$ is basic harmonic, which is not the case in Theorem \ref{codim2One-formalProp}. However, we need the assumption that the metric should be geometrically formal.  
\end{remark}

\subsection{Transverse Lie Foliations}

Transverse foliations form a very particular class of Riemannian foliations which play an important role in the 
Molino structure theorem, cf. \cite{Molino1988_RiemFolsBook}.   For this class of foliations we obtain some very strong results about transverse formality of foliations with dense leaves. 

A foliation $\mathcal{F}$ of codimension $q$ on a manifold $M$ is called a
transverse Lie foliation if its normal bundle $Q$ admits a
global trivialization $X_{1},...,X_{q}$ by foliated vector fields which form
a Lie algebra $\mathbf{g}$. Let $G$ be a simply connected Lie group whose
Lie algebra is isomorphic to $\mathbf{g}$. Let $x_{0}\in M$. Then there
exists a covering $p:\widehat{M}\rightarrow M$ of the manifold $M$, a
surjective submersion $D\colon \widehat{M}\rightarrow G$ of connected
fibres, and a homomorphism $h\colon \pi _{1}(M,x_{0})\rightarrow G$ such
that $D$ is $\pi _{1}(M,x_{0})$-equivariant with respect to the natural
action of $\pi _{1}(M,x_{0})$ on $\widehat{M}$ and the action of $G$ by left
translations. The fibres of $D$ are the leaves of the lifted foliation $%
\widehat{\mathcal{F}}$. In this case, basic forms on $(M,{\mathcal{F}})$
correspond bijectively to $\Gamma =\mathrm{im}\left( h\right) $-invariant
forms on $G$; thus the complex of basic forms $A^{\ast }(M,{\mathcal{F}})$
can be identified with the complex of $\Gamma $-invariant forms $A^{\ast
}(G)^{\Gamma }$. If the leaves of $\mathcal{F}$ are dense in $M$, then the
subgroup $\Gamma $ is dense in $G$. Therefore, if a transverse Lie
foliations has dense leaves, then basic forms are in one-to-one
correspondence with left-invariant forms on $G$.

The normal part of a bundle-like metric $g$ corresponds to a left-invariant
Riemannian metric $\hat{g}$ on $G$. Assume that our foliation is minimal,
i.e. $\kappa =0$. Then any transversely harmonic basic form is closed and
coclosed. The corresponding left-invariant form on $G$ is also closed and
coclosed (and thus harmonic) for the corresponding left-invariant Riemannian
metric on the Lie group $G$.

First, assume that the group $G$ is compact and the Riemannian metric $\hat{g%
}$ is bi-invariant. Then harmonic forms on $G$ are just bi-invariant forms, 
\cite{Hodge1989HarmonicIntegrals, Maurin1997RiemannLegacy}. This fact
permits us to formulate the following proposition.

\begin{proposition}
Let $\mathcal{F}$ be a transverse $\mathbf{g}$-Lie foliation of a compact
manifold $M$ with dense leaves. If the corresponding simply connected Lie
group is compact, then the foliation $\mathcal{F}$ admits a bundle-like
metric for which the foliation is transversely formal.
\end{proposition}

\begin{proof}
Take a bi-invariant Riemannian metric $\hat{g}$ on $G$. We can lift it to a
bundle-like metric on $(M,{\mathcal{F}})$. As $G$ is compact and simply
connected $H^{q}(M,{\mathcal{F}})=H^{q}(G)\neq 0$. Thus the foliation is
taut, and we can modify the metric $\hat{g}$ along the leaves to a minimal
Riemannian metric. Then the basic harmonic forms correspond to bi-invariant
forms on $G$; hence any wedge product of such forms is also basic harmonic.
\end{proof}

Second, consider a transverse Lie foliation modelled on a nilpotent Lie
group $G=N$. For such foliations we have the following result.

\begin{theorem}
\label{nilpotentTheorem}Let $\mathcal{F}$ be a transverse $\mathbf{g}$-Lie
foliation of a compact manifold $M$ with dense minimal leaves. If the Lie
algebra $\mathbf{g}$ is nilpotent with rational structure constants and the
foliation is transversely formal, then $N=\mathbb{R}^{q}$.
\end{theorem}

\begin{proof}
\noindent The basic cohomology $H^{\ast }(M,{\mathcal{F}})$ is isomorphic to
the cohomology $H^{\ast }(\mathbf{n})$ of the Lie algebra $\mathbf{n}=Lie(N)$%
. The module of basic harmonic forms is a minimal model for the cohomology $%
H^{\ast }(\mathbf{n})$. The Lie group $N$ admits a lattice $\Gamma _{0}$,
and the cohomology of the compact nilmanifold $M_{0}=N/\Gamma _{0}$ is also
isomorphic to $H^{\ast }(\mathbf{n})$. Thus the compact nilmanifold $M_{0}$
is formal. The Hasegawa result ensures the nilmanifold $M_{0}$ is a torus,
cf. \cite{Hasegawa1989MinimalModelsNilmflds}. Thus the Lie group $N$ must be
abelian, hence $N=\mathbb{R}^{q}$.
\end{proof}

Lastly, consider a transverse Lie foliation modelled on a solvable Lie group 
$G=S$.
Hisashi Kasuya's results \cite{Kasuya2013GeomFormSolvmflds} on formality of
solvmanifolds can be used to formulate the following result for Lie
foliations.

\begin{theorem}
Let $\mathcal{F}$ be a taut transversely $\mathfrak{g}$-Lie
foliation of a compact manifold $M$ with dense minimal leaves modelled on a
special solvable Lie group $G$, i.e. $G=\mathbb{R}^{k}\times _{\varphi }\mathbb{R}^{s}$ with a
semisimple action $\varphi $. Moreover assume that $G$ admits a lattice .
Then there exists a bundle-like Riemannian metric such that the foliation is
transversely formal for this bundle-like metric.
\end{theorem}

\begin{proof}
Taking into account \cite[Theorem 1.1]{Kasuya2013GeomFormSolvmflds}, the
proof is analogous to that of Theorem \ref{nilpotentTheorem}.
\end{proof}

\vspace{1pt}Next, we show that transverse formality and maximality of the
rank of $H^{1}\left( M,\mathcal{F}\right) $ implies that the manifold has
the structure of an abelian Lie foliation.

Let $(M,{\mathcal{F}},g)$ be an ORFM. Assume that the foliation $\mathcal{F}$
is minimal, i.e. $\kappa =0$. Then the transverse volume form $\nu $ is a
basic harmonic $q$-form. Assume that $\mathcal{F}$ is transversely formal.

Let $\alpha $ and $\beta $ be basic harmonic 1-forms, and let $Y$ be any
vector field tangent to $\mathcal{F}$ . Then

\begin{eqnarray}
0 &=&d\beta (\alpha ^{\sharp },Y)=\alpha ^{\sharp }\beta (Y)-Y\beta (\alpha
^{\sharp })-\beta ([\alpha ^{\sharp },Y])  \notag \\
&=&-Yg(\beta ^{\sharp },\alpha ^{\sharp })-\beta ([\alpha ^{\sharp
},Y])=-\beta ([\alpha ^{\sharp },Y]).  \label{Lie1}
\end{eqnarray}

Let $\alpha, \beta, \gamma $ be basic harmonic 1-forms. Then

\begin{eqnarray}
0 &=&d\gamma (\alpha ^{\sharp },\beta ^{\sharp })  \notag \\
&=&\alpha ^{\sharp }\gamma (\beta ^{\sharp })-\beta ^{\sharp }\gamma (\alpha
^{\sharp })-\gamma ([\alpha ^{\sharp },\beta ^{\sharp }])  \notag \\
&=&\alpha ^{\sharp }g(\gamma ^{\sharp },\beta ^{\sharp })-\beta ^{\sharp
}g(\gamma ^{\sharp },\alpha ^{\sharp })-\gamma ([\alpha ^{\sharp },\beta
^{\sharp }])=-\gamma ([\alpha ^{\sharp },\beta ^{\sharp }]),  \label{Lie2}
\end{eqnarray}

\noindent as $g(\gamma ^{\sharp },\beta ^{\sharp })$ is a constant function,
since $\gamma \wedge \overline{\ast }\beta =g(\gamma ^{\sharp },\beta
^{\sharp })\nu $.

Recall that a codimension $q$ foliation $\left( M,\mathcal{F}\right) $ is
called a $\mathfrak{g}$-Lie foliation if there exist foliated vector fields $%
X_{1},...,X_{q}$ whose projections to $Q$ are linearly independent such that
for any $1\leq i,j\leq q$,  $\left[ X_{i},X_{j}\right] =%
\sum_{k}c_{ij}^{k}X_{k}\,\,\mathrm{mod}\mathcal{F}$, where the $c_{ij}^{k}$  are
the structure constants of the Lie algebra $\mathfrak{g}$. See \cite%
{Fedida1971LieFoliations}.

\begin{proposition}\label{LieFoliationProposition}
Let $(M,{\mathcal{F}},g)$ be an ORFM that is transversely formal with $%
\kappa $ basic harmonic. Suppose $\dim H^{1}(M,{\mathcal{F}})=q=\mathrm{codim%
}\left( \mathcal{F}\right) $, and let $\alpha _{1},...,\alpha _{q}$ be an
orthonormal basis of the space of basic harmonic 1-forms. Then the foliated
vector fields $\alpha _{1}^{\sharp },...,\alpha _{q}^{\sharp }$ commute mod $%
\mathcal{F}$. Therefore the foliation $\mathcal{F}$ is $\mathbb{R}^{q}$-Lie foliation.
\end{proposition}

\begin{proof}
By Theorem \ref{maxRankImpliesMinimalThm}, $\kappa =0$. Observe that $\cap
_{k=1}^{q}\ker \left( \alpha _{k}\right) =T\mathcal{F}$. Since \ref{Lie1}
holds for any leafwise vector field $Y$, each $\alpha _{k}^{\#}$ is a
foliated vector field. By \ref{Lie2}, since $\gamma $ can be chosen to be
any $\alpha _{k}$, each $[\alpha _{j}^{\sharp },\alpha _{m}^{\sharp }]$ must
be a leafwise vector field and is thus  $0\,\,\mathrm{\mathrm{mod}}\mathcal{F}$. 
\end{proof}

\subsection{Transverse formality is not a transverse property}
\label{TransvFormalityNotTransversePropSubsection}

A property of Riemannian foliations $\left( M,\mathcal{F},g\right) $ is
called a \textbf{transverse property (or transverse invariant)} if the
property (or invariant) only depends on the metric induced on $TM\diagup T%
\mathcal{F}$ (the \emph{Riemannian foliation structure}) and not on the
leafwise metric or the choice of orthogonal space $\left( T\mathcal{F}%
\right) ^{\bot }$. Examples of transverse properties and invariants are
tautness, the basic Euler characteristic, eigenvalues of the twisted basic
Laplacian (cf. \cite{HabibRichardson2013_ModifiedDifferentials}). The following demonstrates that transverse formality of $\left( M,%
\mathcal{F},g\right) $ depends on the entire metric $g$ and is not simply
dependent on the Riemannian foliation structure.

\begin{proposition}
Let $\left( M,\mathcal{F},g\right) $ have basic harmonic mean curvature $%
\kappa $. Then for any other metric $g^{\prime }$ that restricts to the same
metric on the normal bundle, the basic projection of the mean curvature form 
$\kappa _{b}^{\prime }$ is basic harmonic if and only if $\kappa
_{b}^{\prime }=\kappa $.
\end{proposition}

\begin{proof}
Let $\delta _{b}$ the basic adjoint of $d$ for the metric $g$; then $\delta
_{b}^{\prime }=\delta _{b}+dh\lrcorner $ is the basic adjoint of $d$ with
respect to the metric $g^{\prime }$ if $\kappa _{b}^{\prime }=\kappa +dh$;
we have that $h$ is a basic function. If $\kappa _{b}^{\prime }$ is basic $%
g^{\prime }$-harmonic, then 
\begin{eqnarray*}
0 &=&\left( d+\delta _{b}^{\prime }\right) \kappa _{b}^{\prime }=\left(
d+\delta _{b}+dh\lrcorner \right) \left( \kappa +dh\right) \\
&=&\left( d+\delta _{b}\right) dh+dh\lrcorner \kappa +dh\lrcorner dh \\
&=&\Delta _{b}h+\left( dh,\kappa \right) +\left\vert dh\right\vert ^{2}.
\end{eqnarray*}%
Integrating over the manifold, we obtain (with $\left\langle \bullet
~,~\bullet \right\rangle $ denoting the $L^{2}$ inner product)%
\begin{eqnarray*}
0 &=&\left\langle \Delta _{b}h,1\right\rangle +\left\langle dh,\kappa
\right\rangle +\left\langle dh,dh\right\rangle \\
&=&\left\langle h,\Delta _{b}1\right\rangle +\left\langle h,\delta
_{b}\kappa \right\rangle +\left\langle dh,dh\right\rangle \\
&=&0+0+\left\Vert dh\right\Vert ^{2},
\end{eqnarray*}%
which can happen only if $h$ is constant and thus $\kappa _{b}^{\prime
}=\kappa $.
\end{proof}

The following theorem demonstrates that on taut foliations, only certain
metrics can be transversely formal, and the choice depends on more than the
transverse metric.

\begin{theorem}
\label{TautCaseKb_zero_Theorem}Let $\left( M,\mathcal{F},g\right) $ be a
foliation of a closed manifold with a bundle-like metric that is
transversely formal. If $\left( M,\mathcal{F}\right) $ is taut, then the
basic component $\kappa _{b}$ of the mean curvature one-form must vanish.
\end{theorem}

\begin{proof}
Suppose that $\left( M,\mathcal{F},g\right) $ is taut and transversely
formal, and suppose that the basic component $\kappa _{b}$ of mean curvature
satisfies $\kappa _{b}=df$ for a basic function $f$. Then the transverse
volume form $\nu $ has length $1$, and $e^{f}\nu $ is basic harmonic, since $%
d\left( e^{f}\nu \right) =0$ and 
\begin{eqnarray*}
\delta _{b}\left( e^{f}\nu \right) &=&\left( -\overline{\ast }d\overline{%
\ast }+df\lrcorner \right) \left( e^{f}\nu \right) =-\overline{\ast }d\left(
e^{f}\right) +e^{f}df\lrcorner \nu \\
&=&-e^{f}\overline{\ast }\left( df\right) +e^{f}df\lrcorner \nu
=-e^{f}df\lrcorner \nu +e^{f}df\lrcorner \nu =0.
\end{eqnarray*}%
Also, $\left\vert e^{f}\nu \right\vert =e^{f}$, and by Lemma \ref%
{constantInnerProdLemma}, $e^{f}$ must be constant, so that $\kappa _{b}=0$.
\end{proof}

\begin{corollary}
\label{q-1Corollary}Let $\left( M,\mathcal{F},g\right) $ be a taut,
codimension $q$ foliation of a closed manifold with a bundle-like metric
that is transversely formal. Then the wedge product of any two basic
harmonic forms is basic harmonic, and  $\dim H^{1}\left( M,\mathcal{F}\right) \neq q-1$.
\end{corollary}

\begin{proof}
By the last theorem, the basic Laplacian $\Delta _{b}$ and the twisted basic
Laplacian  $\Delta _{\kappa _{b}}$  coincide as operators, since $\kappa
_{b}=0$. Then the wedge product of any two basic harmonic forms is basic
harmonic, and $\overline{\ast }$ maps basic harmonic forms to basic harmonic
forms. Then we see that if $\alpha _{1},...,\alpha _{q-1}$ are linearly
independent basic harmonic one-forms, then $\overline{\ast }\left( \alpha
_{1}\wedge ...\wedge \alpha _{q-1}\right) $ is an additional basic harmonic
one-form that is not a linear combination of the others.
\end{proof}

The following    results show that in general the transverse
formality property depends on more than just the transverse metric.

\begin{proposition}
Let $(M,g,\mathcal{F})$ be a nontaut Riemannian, transversely formal foliation that has basic harmonic mean curvature. We assume further that the leaves are not dense, and that the mean curvature has some nonzero component orthogonal to the leaf closures at some point. Then there exists a bundle-like metric $g'$ that induces the same transverse Riemannian structure as $g$, such that $(M,g',\mathcal{F})$ is not transversely formal.
\end{proposition}

\begin{proof} 
With the given assumptions, let $\kappa$ be the basic harmonic mean curvature one-form of $(M,g,\mathcal{F})$, and let $\overline{L}$ denote a leaf closure such that $\kappa^\sharp$ is not contained in $T\overline{L}$ at
some point $x_0$ of $\overline{L}$. Then, for $x\in M$, we let $r=r(x)$ denote the distance from  $x$ to $\overline L$ , and we construct a basic function $f$ as a function of $r$ such that  $f(x)>0$ for $x\in \overline L$ and $f(x)=0$ for $r(x)\ge\epsilon$ for a fixed, sufficiently small $\epsilon>0$. We can then make these choices of $\epsilon$ and $f$ such that $\kappa^\sharp(f)=(\kappa,df)$ is nonzero (and basic).

Next, let $g_L$ and $g_Q$ denote the restrictions of $g$ to $T\mathcal{F}$ and to $Q$, respectively. We define 
the bundle-like metric  $g'=e^{-f}g_L\oplus g_Q$. Now, let $\alpha=\kappa+dh$ be the basic harmonic representative of $\kappa$ in the metric $g'$, for some basic function $h$. We note that $|\kappa|$ is constant with respect to both the metrics $g$ and $g'$. 

Suppose now that $g'$ is also transversely formal, so that $\alpha$ has constant length (with respect to both $g$ and $g'$). Then $|\alpha|^2=|\kappa|^2+|dh|^2+2(\kappa,dh)$ is constant, implying that $|dh|^2+2(\kappa,dh)$ is constant. At a critical point of $h$ this constant is zero, so in fact $|dh|^2=-2(\kappa,dh)$ on $M$. Integrating over $(M,g)$, we see that $\int_M |dh|^2 dv_g=-2\int_M (\kappa,dh)dv_g=-2\int_M (\delta_b\kappa,h)dv_g=0$, so that $h$ is constant, and $\alpha=\kappa$. Since $\delta_b'\alpha=0$, and  $\delta_b'=\delta_b+\frac {\dim\mathcal{F}}2 df\lrcorner$, we have 
$0=\delta_b'\alpha=\delta_b\kappa+\frac {\dim\mathcal{F}}2 (\kappa,df)=\frac {\dim\mathcal{F}}2 (\kappa,df)\neq 0$. This is a contradiction, so we have that $(M,g',\mathcal{F})$ is not transversely formal.
\end{proof}

\begin{remark} Note that by examining the proof above, we can find bundle-like metrics arbitrarily close to transversely formal metrics that are not transversely formal and yet induce the same transverse metric. 
We see a specific example of this situation for the Carri\`{e}re flow $\left( M,\mathcal{F}%
,g\right) $ of Section \ref{Carriere}, where multiplying the stated leafwise metric by a general positive 
basic function causes the transverse formality property to fail.
\end{remark}

The following corollary shows the existence of bundle-like metrics that are not transversely formal on all taut foliations that do not have dense leaves.
\begin{proposition}
Let $(M,g,\mathcal{F})$ be a taut Riemannian, transversely formal foliation. We assume further that the leaves are not dense. Then there exists a bundle-like metric $g'$ that induces the same transverse Riemannian structure as $g$, such that $(M,g',\mathcal{F})$ is not transversely formal.
\end{proposition}

\begin{proof}
By Theorem \ref{TautCaseKb_zero_Theorem}, any taut, transversely formal foliation must have $\kappa_b=0$.
By multiplying the leafwise metric by any nonconstant basic function, we obtain a new bundle-like metric $g'$ with $\kappa_b\neq 0$ and such that $g$ and $g'$ induce the same transverse Riemannian structure. By  Theorem \ref{TautCaseKb_zero_Theorem}, $(M,g',\mathcal{F})$ is not transversely formal.
\end{proof}

If $\left( M,\mathcal{F}\right) $ is a foliation with dense leaves, then
transverse formality is indeed a transverse property.

\begin{proposition}
Let $\left( M,\mathcal{F},g\right) $ be a foliation with dense leaves with
bundle-like metric that is transversely formal. Then any other metric $%
g^{\prime }$ that induces the same transverse metric on $TM\diagup T\mathcal{%
F}$ as $g$ is also transversely formal.
\end{proposition}

\begin{proof}
Assume $\left( M,\mathcal{F},g\right) $ is as given, with $g^{\prime }$ the
new metric. The mean curvature must change by $\kappa _{b}^{\prime }=\kappa
_{b}+dh$ for a basic function $h$, which is necessarily constant, so that $%
\kappa _{b}^{\prime }=\kappa _{b}$.   Note that the basic codifferential $\delta _{b}$ must be the same for both
metrics.   Thus, the basic Laplacian is the same operator on basic forms in
both metrics, and thus the property of transverse formality is preserved.
\end{proof}

\section{Transverse formality for Riemannian flows, $K$-contact and Sasaki
manifolds} \label{FlowSection}

Riemannian flows, that is Riemannian foliations of $1$-dimensional leaves, have been the object of in depth research for several decades due to numerous applications. The rich theory permits us to obtain interesting results on geometric transverse formality as well as to test some general conjectures.   

\begin{notation}
Let $(M^{q+1},\mathcal{F},g)$ be an ORFM. If $\dim \mathcal{F}=1$ with the
foliation be given by the integral curves of a unit vector field $\xi $, we
say that $(M^{q+1},\mathcal{F},g,\xi )$ is a 1-ORFM.  Throughout this section, we will also denote by $H^*_{a\kappa_b}(M,\mathcal{F})$ the basic cohomology of the flow associated with the twisted differential $d_{a\kappa_b}:=d-a\kappa_b\wedge$ for $a\in \mathbb{R}$. 
\end{notation}

\subsection{General results for Riemannian flows}

Recall that Rummler's formula   reduces for Riemannian flows to   %
\begin{equation*}
d\xi ^{\flat }=-\kappa \wedge \xi ^{\flat }+\varphi _{0},
\end{equation*}%
where, as before, $\varphi _{0}=\sum_{j=1}^{q}e^{j}\wedge \nabla _{e_{j}}\xi
^{^{\flat }}\in \Gamma \left( \Lambda ^{2}Q^{\ast }\right) $ is the Euler
form and $\kappa =\left( \nabla _{\xi }\xi \right) ^{^{\flat }}\in \Gamma
\left( Q^{\ast }\right) $ is the mean curvature one-form. Here $%
\{e_{j}\}_{j=1,\ldots ,q}$  is an orthonormal frame of $\Gamma (Q)$ with
corresponding coframe$\{e^{j}\}$.   We have the following crucial lemma
for Riemannian flows. 
\begin{lemma}
\label{EulerFormBasicRiemFlow}(\cite[Lemma 2.4]%
{HabibRichardson2018RiemFlowsAdiabatic}) Let $(M,\mathcal{F},g,\xi )$ be a
1-ORFM. If the mean curvature $\kappa $ of the flow is basic, then $d\varphi
_{0}=-\kappa \wedge \varphi _{0}$ and, thus, the Euler form $\varphi _{0}$
is a basic $2$-form.
\end{lemma}

\begin{proof}
Using Rummler's formula, we compute 
\begin{equation*}
d\varphi_0=d(\kappa\wedge\xi^{\flat })=d\kappa\wedge\xi^{\flat }-\kappa\wedge
d\xi^{\flat }=-\kappa\wedge\varphi_0.
\end{equation*} 
In the last equality, we used the fact that $\kappa$ is a closed form, as it
is basic.   Hence to show that $\varphi_0$ is basic, we compute  
\begin{equation*}
\xi\lrcorner d\varphi_0=-\xi\lrcorner (\kappa\wedge\varphi_0)=0,
\end{equation*}
since $\varphi_0$ is a $2$-form on $Q$. 
\end{proof}

\begin{remark}
The metric may always be chosen so that $\kappa $ is basic, by \cite%
{Dominguez1998FinitenessTensenessThmsRiemFols}. The conclusion of the Lemma
above is in general false if the mean curvature is not basic.
\end{remark}

In the following, we compute the codifferential of the basic form $\varphi
_{0}$, in two different ways. We have

\begin{proposition}
\label{CodifferentialOfEulerFormProposition}Let $(M,\mathcal{F},g,\xi )$ be
a 1-ORFM. Assume that the mean curvature $\kappa $ is basic and harmonic.
Then $\delta _{b}\varphi _{0}=\left( \Delta -\left\vert \varphi
_{0}\right\vert ^{2}-|\kappa |^{2}\right) \xi ^{\flat }$.
\end{proposition}

\begin{proof}
Since $\varphi _{0}$ is a basic form and $\kappa $ is basic, we know by \cite%
[Proposition 2.4]{ParkRichardson1996_BasicLaplacian}, that 
\begin{equation*}
\delta _{b}\varphi _{0}=\delta \varphi _{0}-\varphi _{0}\lrcorner \left( \xi
^{\flat }\wedge \varphi _{0}\right) =\delta d\xi ^{\flat }+\delta (\kappa
\wedge \xi ^{^{\flat }})-\left\vert \varphi _{0}\right\vert ^{2}\xi ^{\flat
}.
\end{equation*}%
As the flow is Riemannian, we have that $\delta \xi ^{\flat }=0$ and, thus, $%
\Delta \xi ^{\flat }=\delta d\xi ^{\flat }$. Now, an easy computation shows
that 
\begin{equation*}
\delta (\kappa \wedge \xi ^{\flat })=(\delta \kappa )\xi ^{\flat }+\nabla^M
_{\xi }\kappa -\nabla^M_{\kappa ^{\#}}\xi ^{\flat }=[\xi ,\kappa
^{\#}]^{^{\flat }}=-|\kappa |^{2}\xi ^{^{\flat }},
\end{equation*}
where we have used that $[\xi ,\kappa ^{\#}]=g(\nabla^M _{\xi }\kappa
^{\#}-\nabla^M_{\kappa ^{\#}}\xi ,\xi )\xi =-|\kappa|^2\xi $, since $%
\kappa ^{\#}$ is basic and that $\delta \kappa =\delta
_{b}\kappa =0$, as $\kappa $ is harmonic. Therefore, $\delta _{b}\varphi
_{0}=\Delta \xi ^{^{\flat }}-(|\kappa |^{2}+|\varphi _{0}|^{2})\xi ^{^{\flat
}}$.
\end{proof}

 One can easily see from Lemma \ref{EulerFormBasicRiemFlow} and Proposition \ref{CodifferentialOfEulerFormProposition}, that for a minimal Riemannian flow, the   $2$-form $\varphi_0$ is a basic harmonic form   if and only if $\xi$ is an eigenvector for the Laplacian $\Delta$ on $M$ associated with the eigenvalue $|\varphi_0|^2$. 
In the following, we denote by $\phi $ the skew-symmetric endomorphism on $%
\Gamma (TM)$ defined by $g(\phi (\cdot ),\cdot )=-g(\nabla^M _{\cdot }\xi
,\cdot )$. A\ simple calculation shows that $\varphi _{0}\left( \cdot ,\cdot
\right) =-2g(\phi (\cdot ),\cdot )$.

\begin{proposition}
\label{EulerFormBasicHarmonicForRiemannianFlows}(in \cite[proof of
Proposition 6.2 for minimal flows]{ElChamiHabib2022BochnerFormRiemFlows})
Let $(M,\mathcal{F},g,\xi )$ be a 1-ORFM, such that the mean curvature is
basic. If $\mathrm{Ric}\left( \xi \right) =\lambda \xi $ for some function $%
\lambda $, then $\delta _{b}\varphi _{0}=-\kappa \lrcorner \varphi _{0}$.
\end{proposition}

\begin{proof}
Fix a point of $M$   and choose a foliated orthonormal
frame $\{e_{i}\}_{i=1,\ldots,q}$ of $\Gamma (Q)$ such that $\nabla e_{i}=0$ at the
point in question.    Recall here that $\nabla$ is the transversal Levi-Civita connection.  For any $Y\in \Gamma (Q)$ parallel at the same point, we compute 

\begin{eqnarray*}
0 &=&\mathrm{Ric}\left( \xi ,Y\right) =\sum_{i=1}^{q}R\left( e_{i},Y,\xi
,e_{i}\right) \\
&=&\sum_{i=1}^{q}g\left( -\nabla^M _{e_{i}}\phi Y+\nabla^M _{Y}\phi e_{i}-\nabla^M
_{\left[ e_{i},Y\right] }\xi ,e_{i}\right) \\
&=&\sum_{i=1}^{q}-g\left( \nabla^M _{e_{i}}\phi Y,e_{i}\right) -g\left( \phi
e_{i},\nabla^M _{Y}e_{i}\right) -g([e_{i},Y],\xi )g\left( \kappa
^{\#},e_{i}\right) ,
\end{eqnarray*}%
  where we use   the fact that $\left[ e_{i},Y\right] $ is leafwise since both $e_{i}$
and $Y$ are parallel at the point in question. Then,   projecting $\nabla^M$ to $Q$, we get that  

\begin{eqnarray*}
0 &=&-\sum_{i=1}^{q}g(\left( \nabla _{e_{i}}\phi \right) \left( Y\right)
,e_{i})-\sum_{i=1}^{q}g(\phi e_{i},\nabla
_{Y}e_{i})-2\sum_{i=1}^{q}g\left( \phi (e_{i}),Y\right) g\left( \kappa
^{\#},e_{i}\right) \\
&=&\frac{1}{%
2}(\delta _{b}\varphi _{0})(Y)+\frac{1}{2}\varphi _{0}(\kappa ^{\#},Y).
\end{eqnarray*}%
Therefore, we deduce that $\delta _{b}\varphi _{0}=-\kappa \lrcorner \varphi
_{0}$.
\end{proof}

Combining Lemma \ref{EulerFormBasicRiemFlow} with Proposition \ref%
{EulerFormBasicHarmonicForRiemannianFlows}, we deduce that for Riemannian
flows with basic mean curvature  if $\mathrm{Ric}\left( \xi \right)
=\lambda \xi $, the Euler form $\varphi _{0}$ is a basic $(-\kappa )$%
-harmonic form. Thus $\varphi _{0}$ is an element of $H_{-\kappa }^{2}(M,%
\mathcal{F})$. Notice here that one can compute the function $\lambda $ explicitly.

\begin{proposition}
Let $(M,\mathcal{F},g,\xi )$ be a 1-ORFM, such that the mean curvature $%
\kappa $ is basic. If $\mathrm{Ric}\left( \xi \right) =\lambda \xi $ for
some function $\lambda $, then $\lambda =-\delta _{b}\kappa +\left\vert \phi
\right\vert ^{2}$.
\end{proposition}

\begin{proof}
We fix a point of $M$ and choose a foliated orthonormal frame $\{e_{i}\}_{i=1,\ldots,q}$ of 
$\Gamma (Q)$ such that  $\nabla e_{i}=0$ at that point.  Then 

\begin{eqnarray*}
\lambda =\mathrm{Ric}\left( \xi ,\xi \right) &=&\sum_{i=1}^{q}R\left(
e_{i},\xi ,\xi ,e_{i}\right) \\
&=&\sum_{i=1}^{q} g\left( \nabla^M_{e_{i}}\nabla^M_{\xi }\xi -\nabla^M_{\xi }\nabla^M
_{e_{i}}\xi -\nabla^M_{\left[ e_{i},\xi \right] }\xi ,e_{i}\right) \\
&=&\sum_{i=1}^{q} g\left( \nabla^M_{e_{i}}\kappa ^{\#},e_{i}\right) +g\left( \nabla^M
_{\xi }\phi e_{i},e_{i}\right) -g\left( \left[ e_{i},\xi \right] ,\xi
\right) g\left( \kappa ^{\#},e_{i}\right) \\
&=&\sum_{i=1}^{q} g\left( \nabla _{e_{i}}\kappa ^{\#},e_{i}\right)-\sum_{i=1}^{q} g\left( \phi e_{i},\nabla^M
_{\xi }e_{i}\right) -\left\vert \kappa \right\vert ^{2} \\
&=&-\delta _{b}\kappa +\left\vert \phi \right\vert ^{2}.
\end{eqnarray*}%
 Here, we used that $\left[ e_{i},\xi \right] $ is leafwise since $e_{i}$ is
foliated. Also, we used that  $\nabla^M_{\xi }e_{i}=-\phi (e_{i})+\left[
e_{i},\xi \right] $. 
\end{proof}

In particular, if $\kappa $ is basic harmonic, then $\lambda =\left\vert
\phi \right\vert ^{2}$. We also notice that the condition $\mathrm{Ric}%
\left( \xi \right) =\lambda \xi $ is not very restrictive since, for
example, it is automatically satisfied for $K$-contact flows on a $\left(
2m+1\right) $-dimensional manifold, with $\lambda =2m$ (see \cite[Theorem 7.1%
]{Blair2010RiemGeomContactSymplMflds}).

\begin{lemma}
\label{lem:flowseuler} (In \cite[Proposition 6.2]%
{ElChamiHabib2022BochnerFormRiemFlows} for minimal flows) Let $(M,\mathcal{F}%
,g,\xi )$ be a 1-ORFM, such that the mean curvature $\kappa $ is basic. If
the Euler class $[\varphi _{0}]\in H_{-\kappa }^{2}(M,\mathcal{F})$ is
nontrivial, then $H_{-\kappa }^{1}\left( M,\mathcal{F}\right) \simeq
H_{-\kappa }^{1}\left( M\right) $, and $1\leq \mathrm{dim}H_{-\kappa
}^{2}\left( M,\mathcal{F}\right) \leq 1+\mathrm{dim}H_{-\kappa }^{2}\left(
M\right)$. When $\mathrm{dim}H_{-\kappa }^{2}\left( M,\mathcal{F}\right)
=1$, we get that $H_{-\kappa }^{2}\left( M\right) \hookrightarrow
H^{1}\left( M,\mathcal{F}\right) $.
\end{lemma}

\begin{proof}
By the Gysin sequence \cite[Proposition 2.30]{RoyoPrieto2004Thesis}, we have%
\begin{equation*}
0\rightarrow H_{-\kappa }^{1}\left( M,\mathcal{F}\right) \rightarrow
H_{-\kappa }^{1}\left( M\right) \rightarrow H^{0}\left( M,\mathcal{F}\right) 
\overset{\wedge \left[ \varphi _{0}\right] }{\rightarrow }H_{-\kappa
}^{2}\left( M,\mathcal{F}\right) \rightarrow H_{-\kappa }^{2}\left( M\right)
\rightarrow H^{1}\left( M,\mathcal{F}\right) .
\end{equation*}

We have $H^{0}\left( M,\mathcal{F}\right) \cong \mathbb{R}$, generated by
constants, and the map $\wedge \left[ \varphi _{0}\right] $ is injective and
thus with one-dimensional range. Then $H_{-\kappa }^{1}\left( M,\mathcal{F}%
\right) \simeq H_{-\kappa }^{1}\left( M\right) $, and the inequality follows
immediately. When $\mathrm{dim}H_{-\kappa }^{2}\left( M,\mathcal{F}\right)
=1$, the last map in the sequence is clearly injective.
\end{proof}

\begin{remark}
One way to ensure that the Euler class is nonzero is to impose the condition 
$\mathrm{Ric}(\xi )=\lambda \xi $ with $\lambda >0$.
\end{remark}

\begin{lemma}
\label{lem:flowseulercohomologybasic} Let $(M,\mathcal{F},g,\xi )$ be a
1-ORFM, such that the mean curvature $\kappa $ is basic.

\begin{enumerate}
\item If $H^{r}\left( M,\mathcal{F}\right) \cong 0$ for some integer $r$,
then 
\begin{equation*}
H_{-\kappa }^{r+2}\left( M,\mathcal{F}\right) \hookrightarrow H_{-\kappa
}^{r+2}\left( M\right) ,
\end{equation*}%
and $\mathrm{dim}H_{-\kappa }^{r+1}\left( M,\mathcal{F}\right) \geq 
\mathrm{dim}H_{-\kappa }^{r+1}\left( M\right) $.  We also have that $$%
\mathrm{dim}H_{-\kappa }^{r+2}\left( M\right) \leq \mathrm{dim}H_{-\kappa
}^{r+2}\left( M,\mathcal{F}\right)+\mathrm{dim}H^{r+1}\left( M,\mathcal{F}%
\right).$$

\item If $H_{-\kappa }^{ r+1}\left( M\right) \cong 0$ and $H_{-\kappa
}^{r+2}\left( M\right) \cong 0$ for some integer $r$, then 
\begin{equation*}
H^{ r}\left( M,\mathcal{F}\right) \simeq H_{-\kappa }^{ r+2}\left( M,\mathcal{F%
}\right).
\end{equation*}
\end{enumerate}
\end{lemma}

\begin{proof}
(1) By the Gysin sequence in \cite[Proposition 2.30]{RoyoPrieto2004Thesis}, 
\begin{equation*}
H_{-\kappa }^{r+1}\left( M,\mathcal{F}\right) \rightarrow H_{-\kappa
}^{r+1}\left( M\right) \rightarrow H^{r}\left( M,\mathcal{F}\right) \overset{%
\wedge \left[ \varphi _{0}\right] }{\rightarrow }H_{-\kappa }^{r+2}\left( M,%
\mathcal{F}\right) \hookrightarrow H_{-\kappa }^{r+2}\left( M\right)
\rightarrow H^{r+1}\left( M,\mathcal{F}\right)
\end{equation*}%
We have that the map $\wedge \left[ \varphi _{0}\right] $ above is the zero
map. Thus the following map is injective. The statement (2) follows from the
same sequence.
\end{proof}

\begin{theorem}
\label{prop:minimalflowformal} Let $(M,\mathcal{F},g,\xi )$ be a 1-ORFM,
such that the flow is minimal. Suppose that $\mathrm{Ric}\left( \xi \right)
=\lambda \xi $ for some function $\lambda>0$. Assume that the flow is
transversely formal and that $H^{r}(M,\mathcal{F})=0$ for some $r$. Then
any basic harmonic $(r+2)$-form is harmonic. If moreover ${\rm dim} H^1(M,\mathcal{F})\geq 1$, then we have 
$${\rm dim} H^{r+1}(M)\leq {\rm dim} H^{r+1}(M,\mathcal{F})\leq  {\rm dim} H^{r+2}(M).$$
\end{theorem}

\begin{proof}
Let $\alpha $ be a basic harmonic $(r+2)$-form. Since the flow is minimal,
the Euler form $\varphi _{0}$ is a basic harmonic $2$-form, by  Proposition  \ref%
{EulerFormBasicHarmonicForRiemannianFlows}. By the transverse formality
assumption, the form $\varphi _{0}\lrcorner \alpha $ is a basic harmonic $r$%
-form and is zero, since $H^{r}(M,\mathcal{F})=0$. Now, 
\begin{equation*}
\delta \alpha =\delta _{b}\alpha +\varphi _{0}\lrcorner (\xi \wedge \alpha
)=\delta _{b}\alpha =0.
\end{equation*}%
Hence $\alpha $ is harmonic.  To prove the second part, we consider the map $H^{r+1}(M,\mathcal{F})\to H^{r+2}(M); \alpha\mapsto\theta^i\wedge\alpha$ where $\{\theta^i\}$ is a basis of $H^1(M,\mathcal{F})$. This map is clearly injective, since $\theta^i\lrcorner\alpha$ is an element in $H^{r}(M,\mathcal{F})=0$. Hence ${\rm dim} H^{r+1}(M,\mathcal{F})\leq  {\rm dim} H^{r+2}(M)$. The other part of the inequality comes from the first part of Lemma \ref{lem:flowseulercohomologybasic}. 
\end{proof}

\subsection{Sasakian manifolds}

We now consider Sasakian manifolds, which are special cases of isometric
flows with very particular transverse geometric structure, namely a K\"{a}%
hler metric.

We will use the following setup. Let $(M^{2n+1},\mathcal{F},g,\xi )$ be a
1-ORFM that is a minimal flow, and we let 
\begin{eqnarray*}
\eta &=&\xi ^{^{\flat }} \\
d\eta &=&\varphi _{0}.
\end{eqnarray*}%
The endomorphism $\phi :TM\rightarrow TM$ and metric $g$ are chosen so that

\begin{eqnarray*}
\varphi _{0}\left( X,Y\right) &=&2g\left( X,\phi Y\right) \\
\phi ^{2} &=&-I+\eta \otimes \xi \\
g\left( \phi X,\phi Y\right) &=&g\left( X,Y\right) -\eta \left( X\right)
\eta \left( Y\right)
\end{eqnarray*}%
for $X,Y\in \Gamma \left( TM\right) $.   The normality condition for Sasakian
manifolds is%
\begin{equation*}
\left( \nabla^M_{X}\phi \right) \left( Y\right) =g\left( X,Y\right) \xi -\eta
\left( Y\right) X.
\end{equation*}%
Note that this condition implies our formula for $\phi $ in the last
section: 
\begin{equation*}
\phi \left( \bullet \right) =-\nabla^M_{\bullet }\xi .
\end{equation*}
 
\begin{remark}
We are using notation consistent with that used in, for example, \cite%
{Fujitani1966ComplexValuedDiffFormsSasaki} and \cite%
{Cappelletti-MontanoEtAl2015SasakianNilmflds}. Other researchers use metrics
that differ from ours by a constant, such as \cite%
{Blair2010RiemGeomContactSymplMflds}.
\end{remark}

\begin{theorem}
\label{SasakiFormalImpliesTransFormalTheorem} Let $\left( M,g\right) $ be a
closed Sasakian manifold of dimension $2n+1$ with Reeb vector field $\xi $.
Let $\mathcal{F}$ be the corresponding Riemannian flow defined by $\xi $. If 
$\left( M,g\right) $ is geometrically formal, then $\left( M,\mathcal{F}%
,g\right) $ is transversely formal.
\end{theorem}

\begin{proof}
Suppose $\left( M,g\right) $ is geometrically formal. By \cite[Theorem 1.2]%
{GrosjeanNagy2009OnTheCohomologyAlgGeomFormal}, $M$ is a rational homology
sphere, so that $H^{r}\left( M\right) \cong 0$ for $1\leq r<n$. By \cite[%
Theorem 7.1]{Blair2010RiemGeomContactSymplMflds}, $K$-contact flows satisfy
the Ricci curvature condition in Lemma \ref%
{EulerFormBasicHarmonicForRiemannianFlows}, and that includes Sasakian
manifolds. Then, by Lemma \ref{lem:flowseuler}, $H^{1}(M,\mathcal{F})\simeq
0 $, and $H^{2}(M,\mathcal{F})\simeq \mathbb{R}$.  Using Lemma \ref%
{lem:flowseulercohomologybasic}, we deduce that $H^{r}(M,\mathcal{F})\simeq
0 $ for $r$ odd and that $H^{r}(M,\mathcal{F})\simeq \mathbb{R}$ for $r$
even. Then $M$ is transversely formal.
\end{proof}

\begin{remark}
\label{HeisenbergRemark}The converse of the statement in Theorem \ref%
{SasakiFormalImpliesTransFormalTheorem} is in general false. From \cite[%
Theorem 1.1]{Cappelletti-MontanoEtAl2015SasakianNilmflds}, any quotient of a
Heisenberg group $H\left( 1,2n\right) $ by a cocompact discrete subgroup $%
\Gamma $ admits a canonical Sasakian structure. In fact, these are the only compact
nilmanifolds that admit Sasakian structures and   are metric fibrations over a flat
torus. Hence, the first cohomology group $H^{1}\left( M\right)$ is isomorphic to $H^1(\mathbb{T}^{2n})\simeq \mathbb{R}^{2n}$. Thus, by \cite[Theorem 1.2]%
{GrosjeanNagy2009OnTheCohomologyAlgGeomFormal}, none of these manifolds are
geometrically formal but they are transversely formal, since the basic harmonic forms are the pullbacks of the harmonic forms on that flat
torus.  
\end{remark}

In the following theorem, we show that transverse formality of closed
Sasakian manifolds implies some formality properties of wedge products of
harmonic forms of specific degrees.

\begin{theorem} \label{SasakiFormalIFFTransFormalTheorem} 
Let $\left( M,g\right) $ be a
closed Sasakian manifold of dimension $2n+1$ with Reeb vector field $\xi $.
Let $\mathcal{F}$ be the corresponding Riemannian flow defined by $\xi $. If 
$\left( M,\mathcal{F},g\right) $ is transversely formal, then
\begin{enumerate}
\item The wedge product of two harmonic forms respectively of degrees less
than $n$ and greater than $n+1$ is harmonic.

\item Assume that  $H^{r}(M,\mathcal{F})=0$ for some integer $r$.  The wedge
product of two harmonic forms respectively of degrees $p$ and $q$ with $%
p+q=r+2$ and $p,q\leq n$ is harmonic.

\item Assume that   $H^{r}(M,\mathcal{F})=0$ for some integer $r$.  The wedge
product of two harmonic forms respectively of degrees $p$ and $q$ with $%
p+q=4n-r$ and $p,q\geq n+1$ is harmonic.
\end{enumerate}
\end{theorem}

\begin{proof}
Let $\alpha $ and $\beta $ be harmonic forms of degree $\leq n$. Then \cite[%
Theorem 4.1]{Tachibana1967OnDecompC-HarmFormsSasaki} implies that $\alpha
,\beta $ are basic forms, and%
\begin{equation*}
\varphi _{0}\lrcorner \alpha =0,~\varphi _{0}\lrcorner \beta =0.
\end{equation*}

Then by \cite[Proposition 2.4]{ParkRichardson1996_BasicLaplacian}, 
\begin{equation*}
\delta _{b}\alpha =\delta \alpha -\varphi _{0}\lrcorner \left( \xi ^{\flat
}\wedge \alpha \right) =\delta \alpha =0.
\end{equation*}%
The same is true for $\beta $, so that $\alpha $ and $\beta $ are basic
harmonic. By transverse formality, $\alpha \wedge \beta $ is basic harmonic,
and also $\alpha \lrcorner \beta $ is basic harmonic. Now, we have 
\begin{equation*}
\delta \left( \alpha \lrcorner \beta \right) =\delta _{b}\left( \alpha
\lrcorner \beta \right) +\varphi _{0}\lrcorner \left( \xi ^{\flat }\wedge
\left( \alpha \lrcorner \beta \right) \right) =\delta _{b}\left( \alpha
\lrcorner \beta \right) =0.
\end{equation*}%
Thus, $\alpha \lrcorner \beta $ is harmonic. Next, suppose $\alpha $ is a
harmonic form of degree $\leq n$, and suppose that $\beta $ is a harmonic
form of degree $\geq n+1$. Then $\ast \beta $ has degree $\leq n$, then $%
\alpha \lrcorner \ast \beta $ is harmonic by the previous argument. But $%
\alpha \lrcorner \ast \beta =\pm \ast \left( \alpha \wedge \beta \right) $,
so $\alpha \wedge \beta $ is harmonic. This proves the first part of the
theorem. To prove the second part, for $\alpha $ and $\beta $ of degrees
less than $n$, the form $\alpha \wedge \beta $ is basic harmonic by the
argument above. If the degree of $\alpha \wedge \beta $ is  $r+2$,  then
Proposition \ref{prop:minimalflowformal} yields the second part. The last
statement follows from the second part by using the Hodge star operator.
\end{proof}

\subsection{Classification of three-dimensional Riemannian flows that are
transversely formal}

In this part, we are going to classify three-dimensional Riemannian flows that are transversely formal. For this, consider a 1-ORFM of three-manifold $M$. If the flow is nontaut, then by Carri\`{e}re
classification \cite[Theorem III.A.1]{Carriere1984FlotsRiem}, it is foliated
diffeomorphic to a hyperbolic torus. The standard metric on these manifolds
is transversely formal (see Section \ref%
{EulerFormBasicHarmonicForRiemannianFlows}).  Hence, we are left with taut Riemannian flows. But Theorem \ref{TautCaseKb_zero_Theorem} tells us that for such flows, transverse formality implies that the basic component of the mean curvature  of the flow vanishes. Again by Carri\`ere classification, the manifold $M$ is diffeomorphic to one of the following: 
\begin{itemize} 
\item $\mathbb{T}^3$ and the flow is generated by a linear flow on $\mathbb{T}^3$.
\item $\mathbb{S}^1\times \mathbb{S}^2$ and the flow is generated by an irrational rotation suspension.
\item $L_{p,q}$ (Lens space) and the flow has 2 closed leaves and is generated by the flow on $\mathbb{S}^3$ as described in \cite{HebdaJ1992ExampleRelevantCurvaturePinching}.
\item Seifert fibration.  
\end{itemize}
 Let us treat each case separately. When $M$ is $\mathbb{T}^3$ (resp. $\mathbb{S}^1\times \mathbb{S}^2$), the formality of the metric in $\mathbb{T}^3$ (resp. $\mathbb{S}^1\times \mathbb{S}^2$) gives transverse formality by Proposition \ref{codim2One-formalProp}. When $M$ is the Lens space, the computations of the mean curvature done in  \cite{HebdaJ1992ExampleRelevantCurvaturePinching} shows that such flow is minimal if and only if it is a Hopf fibration which is not possible because of the closedness of the orbits. We are left with the Seifert fibration. For this particular case, we will assume furthermore that $\mathrm{Ric}(\xi )=\lambda \xi $ with $\lambda
=\left\vert \phi\right\vert ^{2}>0$. By Lemma \ref{EulerFormBasicHarmonicForRiemannianFlows}%
, $\varphi _{0}$ is basic harmonic and transverse formality gives that $\varphi _{0}$
has constant norm, so that $\left\vert\varphi_0\right\vert^2=2\left\vert \phi\right\vert ^{2}$ is a positive constant. Recall here $\phi=-\nabla^M\xi$ is the skew-symmetric endomorphism on $%
\Gamma (TM)$. In dimension three and since $\phi(\xi)=0$, we
can write $\phi$ as the transformation associated to the matrix $\left( 
\begin{array}{cc}
0 & -b \\ 
b & 0%
\end{array}%
\right) $ with $b$ constant. If we rescale the Riemannian metric, the manifold $(M,%
\mathcal{F},b^{2}g,\frac{1}{b}\xi )$ is Sasakian, and the matrix of the new $%
\phi$ is $\left( 
\begin{array}{cc}
0 & -1 \\ 
1 & 0%
\end{array}%
\right) $, a complex structure. From \cite[Theorem 8]%
{BelgunOntheMetricStrNonKahlerCpxSurf} and its restatement in \cite[Section 7%
]{BoyerGalickiMatzeu2006etaEinsteinSasakian}, either $M$ is a quotient of a
round 3-sphere, a Heisenberg 3-manifold, or $\widetilde{SL}(2,\mathbb{R}%
)\diagup \Gamma $ with $\Gamma $ a discrete subgroup of the connected
component of the isometry group of $\widetilde{SL}(2,\mathbb{R})$ in the
natural metric. Observe that quotients of the round sphere and Heisenberg
3-manifolds are necessarily transversely formal, but quotients of the type $%
\widetilde{SL}(2,\mathbb{R})\diagup \Gamma $ have $\mathrm{dim} H^{1}(M,%
\mathcal{F})>2$ and thus are not transversely formal.

\begin{theorem}
Suppose that $(M,\mathcal{F},g,\xi )$ is a 1-ORFM on a 3-manifold. Assume that $M$ is transversely formal. Then $M$ is either $\mathbb{T}^3$, $\mathbb{S}^1\times \mathbb{S}^2$, the hyperbolic torus or a Seifert fibration.  Assume moreover that $\mathrm{Ric}(\xi )=\lambda \xi $ with $%
\lambda>0$, then up to a rescaling of the metric and the flow, $M$ is a quotient of $\mathbb{S}^{3}$ or is a Heisenberg 3-manifold. \label{classificationTheorem}
\end{theorem}

\section{Examples of Riemannian foliations with transverse formality}\label{ExamplesSection}

\subsection{Geometric formality and transverse formality} \label%
{twoKindsFormalitySubsection}

The following examples illustrate that transverse formality and geometric
formality need not coexist for metrics on foliations.

\begin{example} 
In this example, the manifold $\left( M,g\right) $ is formal, but the
Riemannian foliation $\left( M,\mathcal{F},g\right) $ is not transversely
formal. In \cite[Theorem 24]{KotschickTerzic2011_formalityBiquotients},
Kotschick and Terzic showed that the example introduced by Totaro in \cite[%
Section 1]{Totaro2003_CurvatureBiquotients}, which is a biquotient of $\mathbb{S}^{3}\times
\mathbb{S}^{3}\times \mathbb{S}^{3}$ with the standard metric, is not geometrically
formal. This gives a Riemannian foliation by the fibers of this submersion, such that the
total space is formal, but the foliation is not transversely formal.
\end{example}

We already mentioned in Remark \ref{HeisenbergRemark} that the central
foliation for some Heisenberg manifolds is transversely formal, but the
metric on the manifold itself is not formal. In fact, from \cite%
{Hasegawa1989MinimalModelsNilmflds}, the only formal nilmanifolds are tori.
We give another simple example here.

\begin{example}
In this example, the Riemannian foliation $\left( M,\mathcal{F},g\right) $
is transversely formal, but the manifold $\left( M,g\right) $ is not formal
for any bundle-like metric. Let $H$ be a closed hyperbolic surface, which
has first Betti number at least $4$. Let  $M=H\times \mathbb{S}^{1}$,  which has first
Betti number at least $5$. Then the codimension one foliation of $M$ with
leaves of the form $H\times \left\{ \theta \right\} $ is clearly Riemannian
and transversely formal for the product metric. However, by \cite[Theorem 6]%
{Kotschick2001OnProductsHarmFms}, the first Betti number of a geometrically
formal $3$-manifold must be $0$, $1$, or $3$; thus the manifold $M$ is not
formal for any metric.
\end{example}

\subsection{The Hopf fibration of  $\mathbb{S}^{3}$ } \label{HopfFibrationExample}

Consider the sphere $M= \mathbb{S}^{3} \left( r\right) \subseteq \mathbb{R}^{4}$ of
radius $r$ with metric induced from the standard Euclidean metric in
coordinates $\left( x,y,z,w\right) $. We consider the one-dimensional
foliation $\mathcal{F}$ determined by the unit tangent $\xi =\frac{1}{r}%
\left( x\partial _{y}-y\partial _{x}+z\partial _{w}-w\partial _{z}\right) $,
so that the characteristic (leafwise volume) form is $\xi ^{\flat }=\frac{1}{%
r}\left( xdy-ydx+zdw-wdz\right) $, by which we really mean $\frac{1}{r}%
i^{\ast }\left( xdy-ydx+zdw-wdz\right) $, where $i: \mathbb{S}^{3} \left( r\right)
\rightarrow \mathbb{R}^{4}$ is the inclusion. That means we are
considering equivalence classes of forms, so that for example $%
x^{2}+y^{2}+z^{2}+w^{2}=r^{2}$ and $xdx+ydy+zdz+wdw=0$. We consider $e^{1}=%
\frac{1}{r}\left( xdz-zdx+wdy-ydw\right) $ and $e^{2}=\frac{1}{r}\left(
ydz-zdy+xdw-wdx\right) $, unit one-forms on $ \mathbb{S}^{3}\left( r\right)$. We
check that $\xi ^{\flat },e^{1},e^{2}$ is an oriented orthonormal basis of $%
T^{\ast }M$ at each point, and we choose the orientation so that $\ast \xi
^{\flat }=e^{1}\wedge e^{2}$. A computation shows that 
\begin{equation*}
\left( d\xi ^{\flat },e^{1}\wedge e^{2}\right) =\frac{2}{r^{3}}\left(
x^{2}+y^{2}+z^{2}+w^{2}\right) =\frac{2}{r}\text{,}
\end{equation*}%
so that 
\begin{equation*}
d\xi ^{\flat }=\varphi _{0}=\frac{2}{r}e^{1}\wedge e^{2}.
\end{equation*}%
From Rummler's formula, we realize that this means that $\mathcal{F}$ is
minimal (i.e. geodesic) and that the normal bundle is not integrable.  
\begin{eqnarray*}
\Delta \xi ^{\flat } &=&\delta d\xi ^{\flat }=\frac{2}{r}\ast d\ast \left(
e^{1}\wedge e^{2}\right) =\frac{2}{r}\ast d\xi ^{\flat } \\
&=&\frac{4}{r^{2}}\ast \left( e^{1}\wedge e^{2}\right) =\frac{4}{r^{2}}\xi
^{\flat },
\end{eqnarray*}%
so that $\xi ^{\flat }$ (and thus $d\xi ^{\flat }=\varphi _{0}=\frac{2}{r}%
e^{1}\wedge e^{2}$) is an eigenvalue of the Laplacian corresponding to
eigenvalue $\frac{4}{r^{2}}$. We note also that as a basic form, $\varphi
_{0}$ is closed, so that by using the transverse star operator and the fact
that $\mathcal{F}$ is minimal, 
\begin{equation*}
\delta _{b}\varphi _{0}=-\overline{\ast }d\overline{\ast }\left( \varphi
_{0}\right) =-\overline{\ast }d\overline{\ast }\left( \frac{2}{r}e^{1}\wedge
e^{2}\right) =-\frac{2}{r}\overline{\ast }d\xi ^{\flat }=0,
\end{equation*}%
so that $\varphi _{0}$ is basic harmonic. We see that $\left\vert \varphi
_{0}\right\vert ^{2}=\frac{4}{r^{2}}$, a constant. In fact, $\frac{r^{2}}{4}%
\varphi _{0}$ is the transverse volume form of $\left( M,\mathcal{F}\right) $%
.

We see that the basic cohomology $H^{k}\left( M,\mathcal{F}\right) $ is $1$
in degree 0 and 2 (generated by $1$, $\varphi _{0}$, respectively), and
since  $M\diagup \mathcal{F}\cong \mathbb{S}^{2}$,  we get that $H^{1}\left( M,\mathcal{%
F}\right) \cong 0$. Thus, we have that the metric on $M$ is both formal and
transversely formal, and our calculations are consistent with Proposition %
\ref{codim2One-formalProp} and with Theorem \ref%
{SasakiFormalImpliesTransFormalTheorem}, as $r=1$ corresponds to the Sasaki
metric.

\subsection{The Carri\`{e}re example} \label{Carriere}

We will compute the various types of harmonic forms of the Carri\`{e}re
example from \cite{Carriere1984FlotsRiem} in the $3$-dimensional case. Let $%
A $ be a matrix in $\mathrm{SL}_{2}(\mathbb{Z})$ of trace strictly greater
than $2$. We denote respectively by $V_{1}$ and $V_{2}$ the eigenvectors
associated with the eigenvalues $\lambda $ and $\frac{1}{\lambda }$ of $A$
with $\lambda >1$ irrational. Let the hyperbolic torus $M=\mathbb{T}_{A}^{3}$
be the quotient of $\mathbb{T}^{2}\times \mathbb{R}$ by the equivalence
relation which identifies $(m,t)$ to $(A(m),t+1)$. We choose the bundle-like
metric (letting $\left( x,s,t\right) $ denote the local coordinates in the $%
V_{2}$ direction, $V_{1}$ direction, and $\mathbb{R}$ direction,
respectively) as 
\begin{equation*}
g=\lambda ^{-2t}dx^{2}+\lambda ^{2t}ds^{2}+dt^{2}.
\end{equation*}%
First, we notice that the mean curvature of the flow is $\kappa =\kappa
_{b}=\log \left( \lambda \right) dt$, since $\chi _{\mathcal{F}}=\lambda
^{-t}dx$ is the characteristic form and $d\chi _{\mathcal{F}}=-\log \left(
\lambda \right) \lambda ^{-t}dt\wedge dx=-\kappa \wedge \chi _{\mathcal{F}}$%
. Since $\varphi _{0}=0$, the normal bundle is involutive, and we can see
this as the resulting foliation is given by the $\left( y,t\right) $
coordinates. The transverse volume form is $\nu =\lambda ^{t}ds\wedge
dt=d\left( \frac{1}{\log \lambda }\lambda ^{t}ds\right) $, and we see that $%
d\nu =0=\left( d-\kappa \wedge \right) \nu $. Since the flow is nontaut,
this is consistent with the fact that $H^{0}\left( M,\mathcal{F}\right)
\cong H_{\kappa }^{2}\left( M,\mathcal{F}\right) \cong \mathbb{R}$. The
cohomology groups of $M$ satisfy $H^{j}\left( M\right) \cong \mathbb{R}$,
and the generating cohomology classes are generated by $1$, $dt$, $dx\wedge
ds$, $dx\wedge ds\wedge dt$; all are clearly closed. We note that the
codifferential $\delta $ satisfies $\delta =\left( -1\right) ^{k}\ast d\ast $
on $k$-forms, so that we can check if the forms are harmonic by evaluating $%
d\ast $. We see that $d\left( \ast dt\right) =d\left( dx\wedge ds\right) =0$%
, $d\left( \ast \left( dx\wedge ds\right) \right) =d\left( dt\right) =0,$ $%
d\left( \ast \left( dx\wedge ds\wedge dt\right) \right) =d\left( 1\right) =0$%
, so the generators mentioned are all harmonic. The generators are also
constant length, and the wedge product of any two harmonic forms is
harmonic, so that $\left( M,g\right) $ is geometrically formal. Next,
observe that the basic cohomology groups satisfy $H^{0}\left( M,\mathcal{F}%
\right) \cong H^{1}\left( M,\mathcal{F}\right) \cong \mathbb{R}$, with
generators $1$ and $dt$, both clearly closed. Since $\delta _{b}$ is a
restriction of $\delta $ when the mean curvature is basic and the normal
bundle is involutive, we see that both $1$ and $dt$ are basic harmonic
forms. Also, $H_{\kappa }^{1}\left( M,\mathcal{F}\right) =H_{\kappa
}^{2}\left( M,\mathcal{F}\right) \cong \mathbb{R}$, and they are generated
by the basic $\kappa _{b}$-harmonic forms $\overline{\ast }dt=\lambda ^{t}ds$
and $\overline{\ast }1=\nu =\lambda ^{t}ds\wedge dt$. We see that wedge
product of basic harmonic and basic $\kappa $-harmonic forms is basic $%
\kappa $-harmonic, because the only nontrivial case is $\left( dt\right)
\wedge \left( \lambda ^{t}ds\right) =\nu $, so $\left( M,\mathcal{F}%
,g\right) $ is both formal and transversely formal. This example
demonstrates Proposition \ref{codim2One-formalProp}.

\subsection{Solvmanifold admitting transversely formal Riemannian foliations}

We consider the example given in \cite{AndresEtAlExamplesOf4dCpctLCKSolvms}
and in \cite{Banyaga2007exampleLC}, and we construct three Riemannian
foliations on this manifold, all of which are transversely formal. We
consider the Lie group $G_{k}$ of matrices 
\begin{equation*}
\left( 
\begin{array}{cccc}
e^{kz} & 0 & 0 & x \\ 
0 & e^{-kz} & 0 & y \\ 
0 & 0 & 1 & z \\ 
0 & 0 & 0 & 1%
\end{array}%
\right) ,
\end{equation*}%
all variables in $\mathbb{R}$, with $k$ fixed, $e^{k}+e^{-k}=2\cosh \left(
k\right) \in \mathbb{Z}\setminus \left\{ 2\right\} $. This forms a connected
solvable Lie group with basis $\left\{ dx-kxdz,~dy+kydz,~dz\right\} $ of
right-invariant forms.

Let $\Gamma _{k}<G_{k}$ be a discrete subgroup such that $N_{k}=G_{k}\diagup
\Gamma _{k}$ is compact. One such example (adapted from an example in \cite%
{Bock2016OnLowDimSolvMflds}) is defined using $\tau =e^{k}=\frac{3+\sqrt{5}}{%
2};$ $2\cosh k=3$, which satisfies $\tau ^{2}-3\tau +1=0$. We let

\begin{equation*}
\Gamma _{k}=\left\{ \left( 
\begin{array}{cccc}
\tau ^{s} & 0 & 0 & n+m\tau \\ 
0 & \tau ^{-s} & 0 & n+m\tau ^{-1} \\ 
0 & 0 & 1 & s \\ 
0 & 0 & 0 & 1%
\end{array}%
\right) :s,n,m\in \mathbb{Z}\right\}
\end{equation*}%
We check that if $A\left( s,n,m\right) $ is the matrix above, $A\left(
s_{1},n_{1},m_{1}\right) A\left( s_{2},n_{2},m_{2}\right) =A\left(
s_{1}+s_{2},a,b\right) $, where $a$ and $b$ are specific integers, found by
rewriting $\tau ^{s}$ as an integer linear combination of $1$ and $\tau $
and verifying that $\tau ^{-s}$ is the same linear combination of $1$ and $%
\tau ^{-1}$. Thus $\Gamma $ is indeed a discrete subgroup, and it can be
checked that $G_{k}\diagup \Gamma _{k}$ is compact.

The basis $\left\{ dx-kxdz,dy+kydz,dz\right\} $ descends to a global basis
of one-forms $\left\{ \alpha ,\beta ,\gamma \right\} $ over $N_{k}$. The
group $H^{1}\left( N_{k}\right) $ is one dimensional with generator $\left[
\gamma \right] $, and $H^{2}\left( N_{k}\right) $ is also one-dimensional
with generator $\left[ \alpha \wedge \beta \right] $.

Let $\lambda \in \mathbb{R}$ be a number such that $\lambda \left[ \alpha
\wedge \beta \right] \in H^{2}\left( N_{k},\mathbb{Z}\right) $. For nonzero $%
n\in \mathbb{Z}$, let $M_{n,k}$ be the total space of a principal $\mathbb{S}^{1}$%
-bundle over $N_{k}$ with Chern class $n\lambda \left[ \alpha \wedge \beta %
\right] $; such a bundle exists by the result in \cite%
{Kobayashi1956PrincipalFibreBundles}. For simplicity, we denote the
pullbacks of forms on $N_{k}$ to $M_{n,k}$ by the same notations.

Let $\eta $ be a connection form with curvature%
\begin{equation*}
d\eta =n\lambda \alpha \wedge \beta .
\end{equation*}%
The set $\left\{ \alpha ,\beta ,\gamma ,\eta \right\} $ forms a global basis
of one-forms on the $4$-dimensional manifold $M_{n,k}$. We consider the
metric which makes this basis orthonormal.

Relations between these forms from \cite{Banyaga2007exampleLC}:%
\begin{eqnarray}
d\alpha &=&-k\alpha \wedge \gamma  \notag \\
d\beta &=&k\beta \wedge \gamma  \notag \\
d\gamma &=&0  \notag \\
d\eta &=&n\lambda \alpha \wedge \beta.  \label{differentialsOnSolvmanifold}
\end{eqnarray}

\vspace{1pt}We see $d\left( \ast \gamma \right) =d\left( \alpha \wedge \beta
\wedge \eta \right) =d\left( \alpha \wedge \beta \right) \wedge \eta $, and $%
d\left( \alpha \wedge \beta \right) =0$, so that $d\left( \ast \gamma
\right) =0$, so $\gamma $ is the harmonic form that generates $H^{1}\left(
M_{n,k}\right) \cong \mathbb{R}$, and thus $\ast \gamma =\alpha \wedge \beta
\wedge \eta $ is the harmonic form that generates $H^{3}\left(
M_{n,k}\right) \cong \mathbb{R}$. We check that $H^{2}\left( M_{n,k}\right)
\cong 0$, because although $\alpha \wedge \beta $ is closed, $d\left( \frac{1%
}{n\lambda }\eta \right) =\alpha \wedge \beta $, and similarly all other
closed two-forms are exact. Thus, we easily check that the wedge product of
any two harmonic forms is harmonic, so that  $M_{n,k}$  is geometrically
formal.

Let $\omega =k\gamma $ and $\omega ^{\prime }=-k\gamma $. The forms $\alpha
,\alpha \wedge \eta ,\alpha \wedge \gamma ,\alpha \wedge \gamma \wedge \eta $
are $\omega $-harmonic, 
\begin{equation*}
d_{\omega }\left( \frac{1}{2k}\beta \right) =\beta \wedge \gamma ,
\end{equation*}%
so the latter is $d_{\omega }$-exact. In fact, these four $\omega $-harmonic
forms generate all of the cohomology groups $H_{\omega }^{\ast }$. We have $%
d_{\omega }\left( \alpha \wedge \eta \right) =0=\delta _{\omega }\left(
\alpha \wedge \eta \right) $, and also $\beta ,\beta \wedge \eta ,\beta
\wedge \gamma ,\beta \wedge \gamma \wedge \eta $ are $\omega ^{\prime }$%
-harmonic and generate $H_{\omega ^{\prime }}^{\ast }$, and 
\begin{equation*}
d_{\omega ^{\prime }}\left( -\frac{1}{2k}\alpha \right) =\alpha \wedge
\gamma ,
\end{equation*}%
so the latter is $d_{\omega ^{\prime }}$-exact.

The dual vectors to the orthonormal basis of one-forms are (in order) $%
A,B,C,H$. Then we have (for example $n\lambda =$ $d\eta \left( A,B\right)
=-\eta \left( \left[ A,B\right] \right) $)

\begin{eqnarray}
\left[ A,C\right] &=&kA  \notag \\
\left[ A,B\right] &=&-n\lambda H  \notag \\
\left[ B,C\right] &=&-kB  \notag \\
\left[ A,H\right] &=&\left[ B,H\right] =\left[ C,H\right] =0.
\label{LieBracketsSolvmanifold}
\end{eqnarray}

We note that a distribution $L$ of $TM_{n,k}$ is the tangent bundle to a
Riemannian foliation with bundle-like metric if and only if $L$ is
involutive and 
\begin{equation}
\left\langle \left[ X,V\right] ,Y\right\rangle +\left\langle X,\left[ Y,V%
\right] \right\rangle +V\left\langle X,Y\right\rangle =0
\label{Riemannian condition}
\end{equation}%
for all $V\in \Gamma L,$ $X,Y\in \Gamma L^{\bot }$ (see \cite[(5.13),
Theorem 5.17]{Tondeur1997GeometryOfFoliations}). The only possible rank one
distribution with these properties is $L=\mathrm{span}\left\{ H\right\} $.
We check that $\left\langle \left[ A,H\right] ,B\right\rangle +\left\langle
A,\left[ B,H\right] \right\rangle =0$, $\left\langle \left[ A,H\right]
,C\right\rangle +\left\langle A,\left[ C,H\right] \right\rangle =0$, $%
\left\langle \left[ B,H\right] ,C\right\rangle +\left\langle B,\left[ C,H%
\right] \right\rangle =0$, etc., producing a Riemannian flow. For this flow, 
$\eta $ is the characteristic form, and $d\eta =n\lambda \alpha \wedge \beta 
$, so the Euler form is $n\lambda \alpha \wedge \beta $, and the mean
curvature is zero, so this is an isometric flow with normal bundle that is
not involutive.

Now we compute $H^{\ast }\left( M_{n,k},\mathcal{F}\right) $. We have $%
H^{0}\left( M_{n,k},\mathcal{F}\right) \cong \mathbb{R}$ (generated by
constants), $H^{3}\left( M_{n,k},\mathcal{F}\right) \cong \mathbb{R}$
(generated by $\alpha \wedge \beta \wedge \gamma $). From (\ref%
{differentialsOnSolvmanifold}), the only closed one-forms mod exact forms
are multiples of $\gamma $, and $\delta _{b}\gamma =0$, so $\gamma $ is
basic harmonic. Thus $H^{1}\left( M_{n,k},\mathcal{F}\right) \cong \mathbb{R}
$. Then we also have that $H^{2}\left( M_{n,k},\mathcal{F}\right) \cong 
\mathbb{R}$, generated by $\overline{\ast }\gamma =\alpha \wedge \beta $,
which must also be basic harmonic. We see that the wedge product of any
basic harmonic form (linear combination of $1$, $\gamma $, $\alpha \wedge
\beta $, $\alpha \wedge \beta \wedge \gamma $) with any other is also basic
harmonic. Thus, $\left( M_{n,k},\mathcal{F}\right) $ is transversely formal
in this metric.

From (\ref{LieBracketsSolvmanifold}) and (\ref{Riemannian condition}), the
only possible two-dimensional Riemannian foliations are spanned by $\left\{
A,H\right\} $ and $\left\{ B,H\right\} $. We consider the codimension two
foliation with $T\mathcal{F}_{2}=\mathrm{span}\left\{ B,H\right\} $. Then
the characteristic form of this foliation is $\beta \wedge \eta $, and $%
d\left( \beta \wedge \eta \right) =-k\gamma \wedge \left( \beta \wedge \eta
\right) $, so the mean curvature is $\kappa =\kappa _{b}=k\gamma $, which is
not exact, and the normal bundle is involutive. The codimension two
foliation is Riemannian with $H^{0}\left( M_{n,k},\mathcal{F}_{2}\right)
\cong \mathbb{R}$, $H^{2}\left( M_{n,k},\mathcal{F}_{2}\right) \cong 0$,$%
~H_{\kappa }^{0}\left( M_{n,k},\mathcal{F}_{2}\right) \cong 0$, $H_{\kappa
}^{2}\left( M_{n,k},\mathcal{F}_{2}\right) \cong \mathbb{R}$, spanned by the
basic $\kappa $-harmonic volume form $\nu =-\alpha \wedge \gamma =d\left( 
\frac{1}{k}\alpha \right) $. To check $H^{1}\left( M_{n,k},\mathcal{F}%
_{2}\right) $, we see from (\ref{differentialsOnSolvmanifold}) that
multiples of $\gamma $ are the only closed basic forms mod exact forms, and 
\begin{equation*}
\delta _{b}\gamma =-\overline{\ast }d\overline{\ast }\gamma +\kappa
_{b}\lrcorner \gamma =-\overline{\ast }d\left( \alpha \right) +k=-\overline{%
\ast }\left( -k\alpha \wedge \gamma \right) +k=-k+k=0,
\end{equation*}%
so we see that $H^{1}\left( M_{n,k},\mathcal{F}_{2}\right) $ is spanned by $%
\gamma $, which is basic harmonic (and thus $\kappa $ is basic harmonic).
Then we see that $H_{\kappa }^{1}\left( M_{n,k},\mathcal{F}_{2}\right) \cong 
\mathbb{R}$, and this is spanned by the basic $\kappa $-harmonic form $%
\overline{\ast }\gamma =\alpha $. We see in this case that the wedge product
of any basic harmonic form (combination of $1$ and $\gamma $) with any basic 
$\kappa $-harmonic form (combination of $\alpha $ and $\gamma \wedge \alpha $%
) is also basic $\kappa $-harmonic. Therefore, we have transverse formality.
We can see that if we define another codimension two Riemannian foliation by
the span of $\left\{ A,H\right\} $, the results would be similar.

The only codimension three Riemannian foliation has tangent bundle $T%
\mathcal{F}_{3}=\ker \gamma $, which is spanned by $\left\{ A,B,H\right\} $.
Since $\left\vert \gamma \right\vert =1$, the foliation is minimal, which we
can also check by $d\left( \alpha \wedge \beta \wedge \eta \right) =0$. The
basic cohomology injects into ordinary cohomology, and we see that $%
H^{1}\left( M_{n,k},\mathcal{F}_{3}\right) $ is generated by the basic
harmonic form $\gamma $. We clearly have transverse formality in this case
(always true for codimension one foliations).

\section{Appendix}

In this section, we construct a 7-dimensional solvmanifold with an
interesting 3-dimensional nontaut Riemannian foliation with transversely
formal metric and for which the normal bundle is not involutive. Let $%
G=\left\{ M\left( x,y,t,a,b,s,z\right) :\left( x,y,t,a,b,s,z\right) \in 
\mathbb{R}^{7}\right\} $ denote the solvable Lie group of matrices defined
below, which implicitly depends on the nonzero integers $n_{1}$, $n_{2}$ and
nonzero real numbers $k_{1},k_{2}$ that are chosen such that $%
e^{k_{1}}+e^{-k_{1}}=2\cosh \left( k_{1}\right) =K_{1}\in \mathbb{Z}_{\geq 3}
$, $e^{k_{2}}+e^{-k_{2}}=2\cosh \left( k_{2}\right) =K_{2}\in \mathbb{Z}%
_{\geq 3}$. We let $\kappa _{1}=$ $e^{k_{1}}$, $\kappa _{2}=$ $e^{k_{2}}$,
so that each $\kappa _{j}$ satisfies $k^{2}-K_{j}k+1=k^{-2}-K_{j}k^{-1}+1=0$%
. We define the $9\times 9$ matrix $M\left( x,y,t,a,b,s,z\right) $ by 
\begin{eqnarray*}
M\left( x,y,t,a,b,s,z\right) _{11} &=&e^{k_{1}z} \\
M\left( x,y,t,a,b,s,z\right) _{33} &=&e^{-k_{1}z} \\
M\left( x,y,t,a,b,s,z\right) _{44} &=&e^{k_{2}z} \\
M\left( x,y,t,a,b,s,z\right) _{66} &=&e^{-k_{2}z} \\
M\left( x,y,t,a,b,s,z\right) _{22} &=&M\left( x,y,t,a,b,s,z\right) _{55}=1 \\
M\left( x,y,t,a,b,s,z\right) _{77} &=&M\left( x,y,t,a,b,s,z\right) _{88}=1 \\
M\left( x,y,t,a,b,s,z\right) _{21} &=&-n_{1}ye^{k_{1}z} \\
M\left( x,y,t,a,b,s,z\right) _{45} &=&-n_{2}be^{k_{2}z} \\
M\left( x,y,t,a,b,s,z\right) _{\bullet 9} &=&\left( x,t,y,a,s,b,z,0,1\right)
^{T},
\end{eqnarray*}%
with all other entries zero.

A calculation shows that%
\begin{eqnarray*}
&&M\left( x_{1},y_{1},t_{1},a_{1},b_{1},z_{1},s_{1}\right) M\left(
x_{2},y_{2},t_{2},a_{2},b_{2},s_{2},z_{2}\right) \\
&=&M(
x_{1}+x_{2}e^{k_{1}z_{1}},y_{1}+y_{2}e^{-k_{1}z_{1}},t_{1}+t_{2}-n_{1}y_{1}x_{2}e^{k_{1}z_{1}},a_{1}+a_{2}e^{k_{2}z_{1}},\\
&\qquad&\qquad b_{1}+b_{2}e^{-k_{2}z_{1}},s_{1}+s_{2}-n_{2}b_{1}a_{2}e^{k_{2}z_{1}},z_{1}+z_{2}) .
\end{eqnarray*}%
From this we see that the left-invariant vector fields are such that for $%
g\in M$, $\left( L_{g}\right) _{\ast }\partial _{x}=X=e^{k_{1}z}\partial
_{x}-n_{1}ye^{k_{1}z}\partial _{t}$, $\left( L_{g}\right) _{\ast }\partial
_{y}=Y=e^{-k_{1}z}\partial _{y}$, $\left( L_{g}\right) _{\ast }\partial
_{t}=T=\partial _{t}$, $\left( L_{g}\right) _{\ast }\partial
_{a}=A=e^{k_{2}z}\partial _{a}-n_{2}be^{k_{2}z}\partial _{s}$, $\left(
L_{g}\right) _{\ast }\partial _{b}=B=e^{-k_{2}z}\partial _{b}$, $\left(
L_{g}\right) _{\ast }\partial _{s}=S=\partial _{s}$, $\left( L_{g}\right)
_{\ast }\partial _{z}=Z=\partial _{z}$. \newline
The corresponding global dual basis of left-invariant one-forms is $\xi
=e^{-k_{1}z}dx$, $\upsilon =e^{k_{1}z}dy$, $\tau =dt+n_{1}ydx$, $\alpha
=e^{-k_{2}z}da$, $\beta =e^{k_{2}z}db$, $\sigma =ds+n_{2}bda$, $\zeta =dz$.

We have the relations%
\begin{equation}
\begin{array}{ll}
d\xi =k_{1}\xi \wedge \zeta , & d\alpha =k_{2}\alpha \wedge \zeta , \\ 
d\upsilon =-k_{1}\upsilon \wedge \zeta , & d\beta =-k_{2}\beta \wedge \zeta ,
\\ 
d\tau =-n_{1}dx\wedge dy=-n_{1}\xi \wedge \upsilon ,~~ & d\sigma
=-n_{2}da\wedge db=-n_{2}\alpha \wedge \beta , \\ 
d\zeta =0. & 
\end{array}
\label{differentialsNewEx}
\end{equation}

The corresponding bracket relations are:%
\begin{equation*}
\begin{array}{cc}
\left[ X,Z\right] =k_{1}X & \left[ A,Z\right] =k_{2}A \\ 
\left[ Y,Z\right] =-k_{1}Y & \left[ B,Z\right] =-k_{2}B \\ 
\left[ X,Y\right] =-n_{1}T & \left[ A,B\right] =-n_{2}S%
\end{array}%
\end{equation*}%
with all other brackets zero. The differentials of two-forms are as follows:%
\begin{equation}
\begin{array}{c}
\begin{array}{ll}
d\left( \xi \wedge \upsilon \right) =0, & d\left( \alpha \wedge \beta
\right) =0, \\ 
d\left( \upsilon \wedge \tau \right) =-k_{1}\upsilon \wedge \zeta \wedge
\tau , & d\left( \beta \wedge \sigma \right) =-k_{2}\beta \wedge \zeta
\wedge \sigma , \\ 
d\left( \xi \wedge \tau \right) ~~=k_{1}\xi \wedge \zeta \wedge \tau , & 
d\left( \alpha \wedge \sigma \right) =k_{2}\alpha \wedge \zeta \wedge \sigma
, \\ 
d\left( \tau \wedge \zeta \right) =-n_{1}\xi \wedge \upsilon \wedge \zeta ,
& d\left( \sigma \wedge \zeta \right) =-n_{2}\alpha \wedge \beta \wedge
\zeta , \\ 
\multicolumn{1}{c}{d\left( \tau \wedge \alpha \right) =-n_{1}\xi \wedge
\upsilon \wedge \alpha -k_{2}\tau \wedge \alpha \wedge \zeta } & 
\multicolumn{1}{c}{d\left( \xi \wedge \sigma \right) =k_{1}\xi \wedge \zeta
\wedge \sigma +n_{2}\xi \wedge \alpha \wedge \beta } \\ 
\multicolumn{1}{c}{d\left( \tau \wedge \beta \right) =-n_{1}\xi \wedge
\upsilon \wedge \beta +k_{2}\tau \wedge \beta \wedge \zeta } & 
\multicolumn{1}{c}{d\left( \upsilon \wedge \sigma \right) =-k_{1}\upsilon
\wedge \zeta \wedge \sigma +n_{2}\upsilon \wedge \alpha \wedge \beta } \\ 
\multicolumn{1}{c}{d\left( \tau \wedge \sigma \right) =-n_{1}\xi \wedge
\upsilon \wedge \sigma +n_{2}\tau \wedge \alpha \wedge \beta } & 
\multicolumn{1}{c}{d\left( \xi \wedge \zeta \right) =0}%
\end{array}
\\ 
\begin{array}{c}
d\left( \xi \wedge \alpha \right) =k_{1}\xi \wedge \zeta \wedge \alpha -\xi
\wedge k_{2}\alpha \wedge \zeta =\left( -k_{1}-k_{2}\right) \xi \wedge
\alpha \wedge \zeta  \\ 
d\left( \xi \wedge \beta \right) =k_{1}\xi \wedge \zeta \wedge \beta +\xi
\wedge k_{2}\beta \wedge \zeta =\left( -k_{1}+k_{2}\right) \xi \wedge \beta
\wedge \zeta  \\ 
d\left( \upsilon \wedge \alpha \right) =-k_{1}\upsilon \wedge \zeta \wedge
\alpha -\upsilon \wedge k_{2}\alpha \wedge \zeta =\left( k_{1}-k_{2}\right)
\upsilon \wedge \alpha \wedge \zeta  \\ 
d\left( \upsilon \wedge \beta \right) =-k_{1}\upsilon \wedge \zeta \wedge
\beta +\upsilon \wedge k_{2}\beta \wedge \zeta =\left( k_{1}+k_{2}\right)
\upsilon \wedge \beta \wedge \zeta  \\ 
d\left( \upsilon \wedge \zeta \right) =0 \\ 
d\left( \alpha \wedge \zeta \right) =0 \\ 
d\left( \beta \wedge \zeta \right) =0.
\end{array}%
\end{array}  \label{differentialsOf2formsNewEx}
\end{equation}

Let $\Gamma $ be a discrete subgroup of $G$ such that $\Gamma \diagdown G$
is compact. One such example is the set of matrices   
\begin{equation*}
\left( 
\begin{array}{ccccccccc}
\left( \kappa _{1}\right) ^{z} & 0 & 0 & 0 & 0 & 0 & 0 & 0 & x+y\kappa _{1}
\\ 
 A& 1 & 0 & 0 & 0 & 0 & 0 & 0 & t \\ 
0 & 0 & \left( \kappa _{1}\right) ^{-z} & 0 & 0 & 0 & 0 & 0 & x+y\left(
\kappa _{1}\right) ^{-1} \\ 
0 & 0 & 0 & \left( \kappa _{2}\right) ^{z} & 0 & 0 & 0 & 0 & a+b\kappa _{2}
\\ 
0 & 0 & 0 & B& 1 & 0 & 0 & 0 & s \\ 
0 & 0 & 0 & 0 & 0 & \left( \kappa _{2}\right) ^{-z} & 0 & 0 & a+b\left(
\kappa _{2}\right) ^{-1} \\ 
0 & 0 & 0 & 0 & 0 & 0 & 1 & 0 & z \\ 
0 & 0 & 0 & 0 & 0 & 0 & 0 & 1 & 0 \\ 
0 & 0 & 0 & 0 & 0 & 0 & 0 & 0 & 1%
\end{array}%
\right) 
\end{equation*}%
with $A=-n_{1}\left( x \left( \kappa
_{1}\right) ^{z}+y\left( \kappa _{1}\right) ^{z-1}\right)$, 
$B=-n_{2}\left( a+b\left( \kappa _{2}\right) ^{-1}\right) \left(
\kappa _{2}\right) ^{z} $, and 
such that \linebreak $\left( x,y,t,a,b,s,z\right) \in \mathbb{Z}^{7}$.

Using the equations satisfied by $\kappa _{1}$, $\kappa _{2}$, one may
indeed verify that $\Gamma $ is a subgroup. The left-invariant vector fields 
$X,Y,T,A,B,S,Z$ and the left-invariant forms $\xi ,\upsilon ,\tau ,\alpha
,\beta ,\sigma ,\zeta $ descend to well defined one-forms on $\Gamma
\diagdown G$ that satisfy the same relations (\ref{differentialsNewEx}). We
note that none of these forms are exact on the compact manifold, whereas all
were exact on $G$. We choose the metric for which these vector fields form
an orthonormal basis for the tangent space of $\Gamma \diagdown G$.

We consider the three-dimensional distribution spanned by $A$, $S$, and $T$,
which satisfies the Frobenius condition so that it is the tangent space to a
foliation $\mathcal{F}$. We check (\ref{Riemannian condition}) to determine
if the foliation is Riemannian. Since all brackets of basis vectors with $S$
and $T$ are zero, we need only check 
\begin{eqnarray*}
\left\langle \left[ Z,A\right] ,\bullet \right\rangle +\left\langle Z,\left[
\bullet ,A\right] \right\rangle &=&0 \\
\left\langle \left[ B,A\right] ,\bullet \right\rangle +\left\langle B,\left[
\bullet ,A\right] \right\rangle &=&0
\end{eqnarray*}%
where $\bullet $ is any choice of section of $L^{\bot }$. Thus, this
foliation is Riemannian. The characteristic form of the foliation is $\chi_\mathcal{F}
=\alpha \wedge \sigma \wedge \tau $, and we compute%
\begin{eqnarray*}
d\chi_\mathcal{F} &=&d\left( \alpha \wedge \sigma \wedge \tau \right) =d\alpha \wedge
\sigma \wedge \tau -\alpha \wedge d\sigma \wedge \tau +\alpha \wedge \sigma
\wedge d\tau \\
&=&k_{2}\alpha \wedge \zeta \wedge \sigma \wedge \tau -0-n_{1}\alpha \wedge
\sigma \wedge \xi \wedge \upsilon \\
&=&-k_{2}\zeta \wedge \chi_\mathcal{F} -n_{1}\alpha \wedge \sigma \wedge \xi \wedge
\upsilon ,
\end{eqnarray*}%
so that the mean curvature form is $k_{2}\zeta $ and the Euler form is $%
\varphi _{0}=-n_{1}\alpha \wedge \sigma \wedge \xi \wedge \upsilon $, which
is nonzero. Thus, this foliation is nontaut and also as nonintegrable
normal bundle. Also note that $k_{2}\zeta $ is basic, and 
\begin{eqnarray*}
\delta _{b}\left( \kappa _{b}\right) &=&\delta _{b}\left( k_{2}\zeta \right)
=-k_{2}\overline{\ast }d\overline{\ast }\zeta +\kappa _{b}\lrcorner
k_{2}\zeta \\
&=&-k_{2}\overline{\ast }d\left( -\xi \wedge \upsilon \wedge \beta \right)
+k_{2}^{2} \\
&=&k_{2}\overline{\ast }\left( 0\wedge \beta +\xi \wedge \upsilon \wedge
\left( -k_{2}\beta \wedge \zeta \right) \right) +k_{2}^{2} \\
&=&-k_{2}^{2}\overline{\ast }\left( \xi \wedge \upsilon \wedge \beta \wedge
\zeta \right) +k_{2}^{2}=-k_{2}^{2}+k_{2}^{2}=0.
\end{eqnarray*}%
Thus, $\kappa =k_{2}\zeta $ is basic harmonic.

We now compute some of the de Rham cohomology of $\Gamma \diagdown G$ using (%
\ref{differentialsNewEx}) and (\ref{differentialsOf2formsNewEx}). We see
that $H^{1}\left( \Gamma \diagdown G\right) \cong \mathbb{R}$, generated by $%
\left[ \zeta \right] $. The group $H^{2}\left( \Gamma \diagdown G\right)
\cong 0$ (all two closed forms listed in (\ref{differentialsOf2formsNewEx})
are exact). Then  $H^{5}\left( \Gamma \diagdown G\right)\cong H^{2}\left( \Gamma \diagdown G\right)\cong 0$,  and $H^{6}\left( \Gamma
\diagdown G\right) \cong \mathbb{R}$ is generated by $\left[ \xi \wedge
\upsilon \wedge \tau \wedge \alpha \wedge \beta \wedge \sigma \right] $. If
this manifold is indeed formal, then we note that all harmonic $4$-forms $%
\omega $ must be in the kernel of $\zeta \wedge $, because otherwise $\zeta
\wedge $ would yield a nontrivial harmonic $5$-form. But this means all
harmonic $4$-forms are of the form $\zeta \wedge ($harmonic 3-form$)$, which
in turn means each harmonic 3-form is a constant linear combination of the
wedge product of three of the forms $\xi ,\upsilon ,\tau ,\alpha ,\beta
,\sigma $. We note that $\xi \wedge \upsilon \wedge \tau $ and $\alpha
\wedge \beta \wedge \sigma $ are clearly closed and not exact, so that the
rank of $H^{3}\left( \Gamma \diagdown G\right)$ is at least $2$. 

Next, we consider the basic cohomology $H^{\ast }\left( \Gamma \diagdown G,%
\mathcal{F}\right) $, where the foliation has codimension four. The basic
forms must be algebraic combinations of the forms $\xi $, $\upsilon $, $%
\beta $, $\zeta $, because they span the conormal bundle $Q^{\ast }$. We
only need to check if their differentials have no leafwise components. We
see that $d\xi =k_{1}\xi \wedge \zeta $, $d\upsilon =-k_{1}\upsilon \wedge
\zeta $, $d\beta =-k_{2}\beta \wedge \zeta $, and $d\zeta =0$, so only $%
\zeta $ is basic. Since it is also closed and not exact, $H^{1}\left( \Gamma
\diagdown G,\mathcal{F}\right) $ is generated by $\left[ \zeta \right] $. We
have seen above that $\zeta $ is basic-harmonic, since $\delta _{b}\left(
\zeta \right) =0$. Next, we have that locally $\bigwedge^{2}Q^{\ast }$ is
spanned by $\xi \wedge \upsilon ,\upsilon \wedge \beta ,\upsilon \wedge
\zeta ,\xi \wedge \beta ,\beta \wedge \zeta ,\xi \wedge \zeta $. Their
differentials are respectively $0,$ $\left( k_{1}+k_{2}\right) \upsilon
\wedge \beta \wedge \zeta ,0,\left( -k_{1}+k_{2}\right) \xi \wedge \beta
\wedge \zeta ,0,0$, so in fact all of these forms are basic. If $%
k_{1}+k_{2}\neq 0$, $-k_{1}+k_{2}\neq 0$, we see that $\xi \wedge \upsilon $%
, $\upsilon \wedge \zeta $, $\beta \wedge \zeta $, $\xi \wedge \zeta $ are
closed, yet of those $\upsilon \wedge \zeta =d\left( -\frac{1}{k_{1}}%
\upsilon \right) $, $\beta \wedge \zeta =d\left( -\frac{1}{k_{2}}\beta
\right) $, $\xi \wedge \zeta =d\left( \frac{1}{k_{1}}\xi \right) $; $\xi
\wedge \upsilon $ is not the differential of any basic one-form, so that $%
H^{2}\left( \Gamma \diagdown G,\mathcal{F}\right) \cong \mathbb{R}$ with
generator $\xi \wedge \upsilon $. We see also that%
\begin{eqnarray*}
\delta _{b}\left( \xi \wedge \upsilon \right) &=&-\overline{\ast }d\overline{%
\ast }\left( \xi \wedge \upsilon \right) +\kappa _{b}\lrcorner \left( \xi
\wedge \upsilon \right) \\
&=&-\overline{\ast }d\left( \beta \wedge \zeta \right) +0=0,
\end{eqnarray*}%
so that $\xi \wedge \upsilon $ is basic harmonic. If $k_{1}=-k_{2}$, $%
\upsilon \wedge \beta $ is also closed and is not exact, so that $%
H^{2}\left( \Gamma \diagdown G,\mathcal{F}\right) \cong \mathbb{R}^{2}$,
spanned by $\xi \wedge \upsilon $ and $\upsilon \wedge \beta $. In this case,%
\begin{eqnarray*}
\delta _{b}\left( \upsilon \wedge \beta \right) &=&-\overline{\ast }d%
\overline{\ast }\left( \upsilon \wedge \beta \right) +\kappa _{b}\lrcorner
\left( \upsilon \wedge \beta \right) \\
&=&-\overline{\ast }d\left( \xi \wedge \zeta \right) +0=0,
\end{eqnarray*}%
and also $\upsilon \wedge \beta $ is basic harmonic. If $k_{1}=k_{2}$, $\xi
\wedge \beta $ is also closed and is not exact, so that $H^{2}\left( \Gamma
\diagdown G,\mathcal{F}\right) \cong \mathbb{R}^{2}$, spanned by $\xi \wedge
\upsilon $ and $\xi \wedge \beta $. In this case,%
\begin{eqnarray*}
\delta _{b}\left( \xi \wedge \beta \right) &=&-\overline{\ast }d\overline{%
\ast }\left( \xi \wedge \beta \right) +\kappa _{b}\lrcorner \left( \xi
\wedge \beta \right) \\
&=&-\overline{\ast }d\left( -\upsilon \wedge \zeta \right) +0=0,
\end{eqnarray*}%
so also $\xi \wedge \beta $ is basic harmonic.

The three forms in $\bigwedge^{3}Q^{\ast }$ are spanned locally by $\upsilon
\wedge \beta \wedge \zeta $, $\xi \wedge \beta \wedge \zeta $, $\xi \wedge
\upsilon \wedge \zeta $, $\xi \wedge \upsilon \wedge \beta $. We see that
their differentials are, respectively, $0$, $0$, $0$, $-k_{2}\xi \wedge
\upsilon \wedge \beta \wedge \zeta $. We see that $\upsilon \wedge \beta
\wedge \zeta =d\left( \frac{1}{k_{1}+k_{2}}\upsilon \wedge \beta \right) $
if $k_{1}\neq -k_{2}$, $\xi \wedge \beta \wedge \zeta =d\left( \frac{1}{%
-k_{1}+k_{2}}\xi \wedge \beta \right) $ if $k_{1}\neq k_{2}$, and finally $%
\xi \wedge \upsilon \wedge \zeta $ is not the differential of any basic
form. So if $k_{1}\neq -k_{2}$ and $k_{1}\neq k_{2}$, $H^{3}\left( \Gamma
\diagdown G,\mathcal{F}\right) \cong \mathbb{R}$, generated by $\xi \wedge
\upsilon \wedge \zeta $. Note 
\begin{eqnarray*}
\delta _{b}\left( \xi \wedge \upsilon \wedge \zeta \right)  &=&-\overline{%
\ast }d\overline{\ast }\left( \xi \wedge \upsilon \wedge \zeta \right)
+\kappa _{b}\lrcorner \left( \xi \wedge \upsilon \wedge \zeta \right)  \\
&=&-\overline{\ast }d\left( -\beta \right) +k_{2}\xi \wedge \upsilon =-k_{2}%
\overline{\ast }\left( \beta \wedge \zeta \right) +k_{2}\xi \wedge \upsilon 
\\
&=&-k_{2}\overline{\ast }\left( \beta \wedge \zeta \right) +k_{2}\xi \wedge
\upsilon =-k_{2}\xi \wedge \upsilon +k_{2}\xi \wedge \upsilon =0,
\end{eqnarray*}%
so this form is basic harmonic. If $k_{1}=k_{2}$, $\xi \wedge \beta \wedge
\zeta $ generates another cohomology class, and 
\begin{eqnarray*}
\delta _{b}\left( \xi \wedge \beta \wedge \zeta \right)  &=&-\overline{\ast }%
d\overline{\ast }\left( \xi \wedge \beta \wedge \zeta \right) +\kappa
_{b}\lrcorner \left( \xi \wedge \beta \wedge \zeta \right)  \\
&=&-\overline{\ast }d\upsilon +k_{2}\xi \wedge \beta =-\overline{\ast }%
\left( -k_{1}\upsilon \wedge \zeta \right) +k_{2}\xi \wedge \beta  \\
&=&k_{1}\left( -\xi \wedge \beta \right) +k_{2}\xi \wedge \beta =0,
\end{eqnarray*}%
so that form is also basic harmonic. If $k_{1}=-k_{2}$, $\upsilon \wedge
\beta \wedge \zeta $ generates another cohomology class, and%
\begin{eqnarray*}
\delta _{b}\left( \upsilon \wedge \beta \wedge \zeta \right)  &=&-\overline{%
\ast }d\overline{\ast }\left( \upsilon \wedge \beta \wedge \zeta \right)
+\kappa _{b}\lrcorner \left( \upsilon \wedge \beta \wedge \zeta \right)  \\
&=&-\overline{\ast }d\left( -\xi \right) +k_{2}\upsilon \wedge \beta =-%
\overline{\ast }\left( -k_{1}\xi \wedge \zeta \right) +k_{2}\upsilon \wedge
\beta  \\
&=&k_{1}\upsilon \wedge \beta +k_{2}\upsilon \wedge \beta =0,
\end{eqnarray*}%
and that form is also basic harmonic. In either of these last two cases, $%
H^{3}\left( \Gamma \diagdown G,\mathcal{F}\right) \cong \mathbb{R}^{2}$.
Finally, $H^{4}\left( \Gamma \diagdown G,\mathcal{F}\right) \cong 0$ since
the foliation is nontaut. We now calculate the $\kappa $-harmonic forms,
using the differential $d-\kappa \wedge $. We have $H_{\kappa }^{0}\left(
\Gamma \diagdown G,\mathcal{F}\right) \cong 0$ since $\kappa $ is not basic
exact. Among the one-forms, we see that%
\begin{eqnarray*}
d\xi  &=&k_{1}\xi \wedge \zeta ,d\upsilon =-k_{1}\upsilon \wedge \zeta
,d\beta =-k_{2}\beta \wedge \zeta ,d\zeta =0 \\
\kappa \wedge \xi  &=&-k_{2}\xi \wedge \zeta ,\kappa \wedge \upsilon
=-k_{2}\upsilon \wedge \zeta ,\kappa \wedge \beta =-k_{2}\beta \wedge \zeta
,\kappa \wedge \zeta =0.
\end{eqnarray*}%
From the above, $\zeta $ is basic $\kappa $-exact, since $\left( d-\kappa
\wedge \right) \left( -\frac{1}{k_{2}}\right) =\zeta $, and $\beta $ is
basic $\kappa $-closed and not basic $\kappa $-exact, and $\xi $ is basic $%
\kappa $-closed if and only if $k_{1}=-k_{2}$, and $\upsilon $ is basic $%
\kappa $-closed if and only if $k_{1}=k_{2}$. In all of these cases, these
generating forms are basic $\kappa $-harmonic. We could also tell this since 
$\overline{\ast }\delta _{b}=\pm \left( d-\kappa _{b}\wedge \right) 
\overline{\ast }$ and $\overline{\ast }d=\pm \left( \delta _{b}-\kappa
_{b}\lrcorner \right) \overline{\ast }~$on basic forms. In summary, $%
H_{\kappa }^{1}\left( \Gamma \diagdown G,\mathcal{F}\right) \cong \mathbb{R}$
unless $k_{1}=\pm k_{2}$, in which case $H_{\kappa }^{1}\left( \Gamma
\diagdown G,\mathcal{F}\right) \cong \mathbb{R}^{2}$. Similarly, we see that $%
H_{\kappa }^{2}\left( \Gamma \diagdown G,\mathcal{F}\right) \cong \mathbb{R}$%
, generated by the basic $\kappa $-harmonic form $\overline{\ast }\left( \xi
\wedge \upsilon \right) =\beta \wedge \zeta $, if $k_{1}\neq \pm k_{2}$. If $%
k_{1}=k_{2}$, we have the additional basic $\kappa $-harmonic form $%
\overline{\ast }\left( -\xi \wedge \beta \right) =\upsilon \wedge \zeta $,
and if $k_{1}=-k_{2}$, we have $\overline{\ast }\left( \upsilon \wedge \beta
\right) =\xi \wedge \zeta $, also basic $\kappa $-harmonic. Thus, $H_{\kappa
}^{2}\left( \Gamma \diagdown G,\mathcal{F}\right) \cong \mathbb{R}$ unless $%
k_{1}=\pm k_{2}$, in which case $H_{\kappa }^{2}\left( \Gamma \diagdown G,%
\mathcal{F}\right) \cong \mathbb{R}^{2}$. Finally, $H_{\kappa }^{3}\left(
\Gamma \diagdown G,\mathcal{F}\right) $ is generated by $-\overline{\ast }%
\zeta =\xi \wedge \upsilon \wedge \beta $, also basic $\kappa $-harmonic. We
also observe that in all these cases, the wedge product of a basic harmonic
form and a basic $\kappa $-harmonic form is basic $\kappa $-harmonic,
verifying that $\left( \Gamma \diagdown G,\mathcal{F}\right) $ with the
given metric is transversely formal.

\bibliographystyle{amsplain}
\bibliography{habibRichardsonWolak}

\end{document}